\documentclass[11pt]{amsart}

\usepackage{biblatex}

\usepackage{hyperref}
\usepackage[margin=1in]{geometry} 
\usepackage{amsmath,amsthm,amssymb, mathtools}
\usepackage{cleveref}
\usepackage{graphicx}
\usepackage{biblatex}
\usepackage{tikz-cd}
\usepackage{mathrsfs}
\usepackage{todonotes}
\usepackage{xcolor}
\usepackage[new]{old-arrows}

\newtheorem{theorem}{Theorem}[section] 
\newtheorem{proposition}[theorem]{Proposition}

\newtheorem{corollary}[theorem]{Corollary}
\newtheorem{lemma}[theorem]{Lemma}

\theoremstyle{definition}
\newtheorem{definition}[theorem]{Definition}
\newtheorem{example}[theorem]{Example}

\theoremstyle{remark}
\newtheorem{remark}[theorem]{Remark}

\newtheorem{joke}[theorem]{Joke}

\newcommand{\overbar}[1]{\mkern1.5mu\overline{\mkern-1.5mu#1\mkern-1.5mu}\mkern 1.5mu}

\title{Generalized Sums of Linear Orders}

\author{\'Alvaro D\'iaz Ramos}
\address[D\'iaz Ramos]{
Department of Mathematics, 
Stanford University
}
\email{aldr@stanford.edu}

\author{Garrett Ervin}
\address[Ervin]{
Division of Physics, Mathematics, and Astronomy, 
California Institute of Technology
}
\email{gervin@caltech.edu}

\author{Saharon Shelah}
\address[Shelah]{
Einstein Institute of Mathematics, 
Hebrew University of Jerusalem
}
\email{shelah@math.huji.ac.il}

\begin{document}

\begin{abstract}
We study generalized sums of linear orders. These are binary operations that, given linear orders $A$ and $B$, return an order $A \oplus B$ that can be decomposed as an isomorphic copy of $A$ interleaved with a copy of $B$. We show that there is a rich array of associative sums different from the usual sum $+$ and its dual. The simplest of these sums arise from what we call \textit{sum-generating classes} of linear orders. These classes determine canonical decompositions of every linear order into left and right halves. We study the structural and algebraic properties of these classes along with the sums they generate.

We then turn our attention to commutative sums on various subclasses of the linear orders. For this, we introduce the notion of a \textit{complicated class} of linear orders and show that over such classes sums can be constructed in a very flexible way. Using this construction, we prove the existence of associative sums lacking the structural properties of the usual sum. Along the way, we characterize the associative and commutative sums on the ordinals.
\end{abstract}
\maketitle

\tableofcontents
\newpage

\section{Introduction} \label{sect:intro}

The usual sum $A + B$ of two linear orders $A$ and $B$ is obtained by concatenating $A$ and $B$. This operation resembles the disjoint union of sets in that $A + B$ can be partitioned (by cutting at the $+$ sign) into a suborder isomorphic to $A$ and a suborder isomorphic to $B$. We call a binary operation $\oplus$ on the class of linear orders $LO$ a \emph{sum} if it shares this property with $+$, that is, if for every $A, B \in LO$, the order $A \oplus B$ can be written as a disjoint union $A' \sqcup B'$ of two suborders $A'$ and $B'$ such that $A' \cong A$ and $B' \cong B$. In this paper, we are interested in studying such operations. 

The usual sum also has the following properties:
\begin{itemize}
	\item[(i.)] (\emph{regularity}) if $A \cong X$ and $B \cong Y$, then $A + B \cong X + Y$,
	\item[(ii.)] (\emph{associativity}) $A + (B + C) \cong (A + B) + C$. 
\end{itemize}

Dual to the usual sum is the reverse operation $+^*$ defined by $A +^* B = B + A$. Like the usual sum, its dual is regular and associative. 

While intuitively natural operations on linear orders, $+$ and $+^*$ are not sums in the categorical sense of being coproducts. In particular, these operations are not commutative: there are many pairs of orders $A, B$ for which $A + B \not\cong B + A$, and hence also $B +^* A \not\cong A +^* B$. One might ask whether $+$ and $+^*$ have a natural intrinsic characterization among the sums on $LO$. More concretely, we consider the following questions:

\begin{itemize}
	\item[(1.)] Are $+$ and $+^*$ the only regular, associative sums on $LO$?
	\item[(2.)] If not, can $+$ and $+^*$ be characterized among the regular, associative sums on $LO$?
\end{itemize}

The answer to (1.) turns out to be negative in a very strong sense, and one of our main goals in this paper is to construct many regular and associative sums $\oplus$ on $LO$ and natural subclasses of $LO$ that differ from the usual sum and its dual. 

We approach the construction of such sums in two different ways. In the first approach, we begin by defining a global scheme on $LO$ that decomposes every linear order into convex pieces. To sum a given pair of orders $A$ and $B$, each individual $A$-piece is summed with a corresponding $B$-piece with one of $+$ of $+^*$, and the resulting orders are then concatenated to form $A \oplus B$. The regularity and associativity of sums defined in this way will follow from certain canonical properties of the decomposition scheme. 

The simplest scheme is the one under which every linear order is decomposed into a single convex piece (i.e., not decomposed at all). Sums $\oplus$ defined relative to such a scheme have the property that for every $A$ and $B$, either $A \oplus B = A + B$ or $A \oplus B = A +^* B$. A natural question is whether both possibilities can be realized by a single sum, that is, whether it can be that for certain pairs of inputs $A$ and $B$, $\oplus$ returns their usual sum, and for others, their dual sum. It turns out that the answer is negative: we show in \Cref{th: semi-standard sums are standard} that the only regular, associative such sums are $+$ and $+^*$. Said another way, \Cref{th: semi-standard sums are standard} shows that if $\oplus$ is a regular, associative sum on $LO$ such that $A \oplus B$ can always be decomposed into \textit{convex} copies of $A$ and $B$, then $\oplus$ is one of $+$ and $+^*$. This can be viewed as an answer to question (2.) above. 

On the other hand, there are decomposition schemes yielding novel sums under which every order is decomposed into at most \textit{two} convex pieces. Such schemes cut every linear order $A$ ``in the middle," that is, into an initial segment $A_L$ and corresponding final segment $A_R$, so that $A \cong A_L + A_R$. We will focus on the case where the cut made is determined by what we call a \textit{sum-generating class} of linear orders; see \Cref{def: left sum-generating class}. For a fixed sum-generating class $\mathcal{C}$, an associated decomposition scheme is defined by taking $A_L$ to be the longest initial segment of a given order $A$ that belongs to $\mathcal{C}$, and $A_R$ the remaining final segment. We then define a sum relative to this scheme by setting 
\[
A \oplus B := (A_L \circ B_L) + (A_R \diamond B_R),
\] 
where $\circ$ is chosen to be uniformly either $+$ or $+^*$, and likewise for $\diamond$. 

The defining properties of sum-generating classes ensure that sums defined in this way are regular and associative. Conversely, we prove that if a two-piece decomposition scheme is determined by a class of orders $\mathcal{C}$ in the above sense and generates a regular and associative sum, then $\mathcal{C}$ must be a sum-generating class. See \Cref{th: characterization of sum-generating classes with simple sums}.

Sum-generating classes and the decompositions they determine have interesting algebraic and structural properties that we examine in \Cref{subsect: algebraic properties}. In \Cref{subsect: filtrations and sifted sums}, we iterate the two-piece construction to get sums determined by decompositions with many (even infinitely many) pieces. 

Our second approach to constructing sums is inductive. Rather than depending on a global structural decomposition theory like the sums that arise from sum-generating classes, this approach depends instead on a \textit{non-structure} property for the class of orders over which the sum is being defined. More precisely, we begin with a class $K$ of linear orders of a fixed cardinality $\lambda$ with the property that any two orders $A, B \in K$ can be summed together in the maximum possible number of ways, i.e.
\[
|\{X \in K: X = A' \sqcup B', A' \cong A, B' \cong B\}/\cong| = 2^{\lambda}.
\]
We call such $K$ \textit{complicated classes}. In \Cref{subsect: compliacted classes}, we show how to inductively construct sums $\oplus$ on complicated classes $K$ that are not only regular and associative, but commutative as well. We call such sums \textit{good}; see \Cref{sect: good sums of linear orders}. We show how this construction can be used to get good sums on certain classes of non-scattered orders; see \Cref{ex: binary shuffle complicated class}, \Cref{ex: usable compliacted class}.

In the course of our work we also investigate good sums over certain highly structured classes of linear orders. In \Cref{subsect: filtrations and sifted sums}, we show how the Hessenberg sum on the ordinals arises by way of an iterated sum-generating class construction. In \Cref{subsect: good sums of well-orders}, we investigate good sums of well-orders, ultimately giving bounds on the possible sums of two ordinals; see \Cref{th: min sum is minimal/bounds on instances of a sum of ordinals}. In \Cref{subsect: rational shuffles}, we construct a good sum on the class of countable rational shuffles that is quite different from any of the other sums considered in the paper. 

The culmination of our work is \Cref{sect: revisiting the properties of sums}, where we prove that there exist regular, associative sums lacking the basic structural properties of the usual sum. To achieve this, we combine our work on good sums on the ordinals, complicated classes, and filtrations of sum-generating classes.

\section{Conventions and preliminaries}\label{sect: preliminaries}
A \emph{linear order} is a pair $(X, <)$ consisting of a set $X$ and a strict total ordering $<$ on $X$. We often suppress mention of the ordering relation and denote $(X, <)$ by $X$. To indicate that an ordering relation $<$ is an ordering on $X$, we sometimes write it as $<_X$. 

Given a linear order $(X, <)$, its \emph{reverse} is the order $(X, <^*)$, where the ordering $<^*$ is defined by $x <^* y$ if and only if $y < x$, for all $x, y \in X$. We also use the shorthand $X^*$ for this order. 

A \emph{suborder of $X$} is a subset $Y \subseteq X$ equipped with the inherited ordering from $X$. We use the terms \emph{interval} and \emph{convex subset} of $X$ interchangeably to mean a suborder $I \subseteq X$ such that for all $x, y, z \in X$, whenever $x, y \in I$ and $x < z < y$, then $z \in I$. 

An \emph{initial segment} of $X$ is a convex subset $L$ of $X$ with the property that if $x < y$ and $y \in L$ then $x \in L$. A subset $R \subseteq X$ is a \emph{final segment} of $X$ if $X \setminus R$ is an initial segment; equivalently, if whenever $x > y$ and $y \in R$, then $x \in R$. 

An \emph{embedding} from a linear order $X$ to a linear order $Y$ is a map $f: X \rightarrow Y$ such that $x < y$ implies $f(x) < f(y)$ for all $x, y \in X$. We also call embeddings \emph{order-preserving maps}. Embeddings are automatically injective. A \emph{convex embedding} is an embedding $f: X \rightarrow Y$ such that $f[X]$ is convex in $Y$. An \emph{isomorphism} is a surjective embedding. We write $f: X \xrightarrow{\sim} Y$ to indicate that $f$ is an isomorphism from $X$ to $Y$. 

We write $X \preceq Y$ if there is an embedding from $X$ to $Y$, $X \preceq_c Y$ if there is a convex embedding from $X$ to $Y$, and $X \cong Y$ if there is an isomorphism from $X$ to $Y$. We write $X \equiv Y$ if $X \preceq Y$ and $Y \preceq X$, and in this case say that $X$ and $Y$ are \emph{equimorphic}. Equimorphic orders need not be isomorphic in general. 

An \emph{order type} is an isomorphism class of linear orders. Though throughout the paper we will work with orders as opposed to order types, we are often interested in a given order $X$ only up to isomorphism, and we may occasionally conflate $X$ with its order type. 

$LO$ denotes the class of all linear orders. We emphasize that the members of $LO$ are specific linear orders and not order types. We will study various operations on $LO$, and frequently (though not always) these operations will be invariant under isomorphism, in the sense that isomorphic inputs yield isomorphic outputs. Such operations can also be viewed as operations on order types. 

A \emph{well-order} is a linear order $X$ such that every non-empty subset $Y \subseteq X$ has a $<_X$-least element. An \emph{ordinal} is a transitive set well-ordered by the set-membership relation $\in$. We will assume that the reader is familiar with ordinals and their basic properties. For ordinals $\alpha$ and $\beta$, we use the conventional notation $\alpha < \beta$ to mean $\alpha \in \beta$ and $\alpha \leq \beta$ for $\alpha < \beta \vee \alpha = \beta$. For ordinals we have $\alpha \leq \beta$ if and only if $\alpha \leqslant \beta$, and $\alpha < \beta$ if and only if $\alpha \leqslant \beta$ and $\beta \not\leqslant \alpha$. We use $Ord$ to denote the class of ordinals, and $Lim$ to denote the subclass of limit ordinals. 

A \emph{sequence} is a set $\{x_{\beta}: \beta < \alpha\}$ indexed by an ordinal $\alpha$. The \emph{length} of such a sequence is $\alpha$. 

As usual, we identify finite ordinals with natural numbers: $0 = \emptyset$, $1 = \{0\}$, $2 = \{0, 1\}$, and so on. We write $\omega$ for the set of natural numbers $\{0, 1, \ldots\}$, which is also the least infinite ordinal. A \emph{cardinal} is an ordinal that is not in bijection with any smaller ordinal. For $\alpha$ an ordinal, we use both $\aleph_{\alpha}$ and $\omega_{\alpha}$ to denote the $\alpha$-th infinite cardinal. 

Suppose that $X$ is a linear order, and for every $x \in X$ we fix a linear order $A_x$. The \emph{ordered sum} $\Sigma_{x \in X} A_x$ is the set of ordered pairs $\{(x, a): x \in X, a \in A_x\}$ equipped with the ordering relation $<$ defined by lexicographic rule $(x, a) < (y, b)$ if either $x <_X y$, or $x = y$ and $a <_{A_x} b$. Visually, $\Sigma_{x \in X} A_x$ is the linear order obtained by replacing each point $x \in X$ with a copy of $A_x$. We will also call an order of the form $\Sigma_{x \in X} A_x$ an \emph{$X$-sum}. 

If there is a linear order $Y$ such that $A_x = Y$ for every $x \in X$, then we denote the sum $\Sigma_{x \in X} Y$ by $X \times Y$ and call this order the \emph{lexicographic product} of $X$ and $Y$. 

If $A$ and $B$ are linear orders, the \emph{sum} of $A$ and $B$, denoted $A + B$, is the $2$-sum $\Sigma_{x \in 2} A_x$ where $A_0 = A$ and $A_1 = B$. Visually, $A + B$ is the order obtained by placing a copy of $B$ to the right of $A$. Up to isomorphism, $A + B$ is the unique linear order with an initial segment isomorphic to $A$ whose corresponding final segment is isomorphic to $B$. We write $+^*$ for the dual operation, defined by $A +^* B = B + A$. Together, $+$ and $+^*$ are the \emph{standard sums} on $LO$. 

Both the sum and its dual are associative in the sense that for all $A, B, C \in LO$ we have
\[
\begin{array}{r c l}
(A + B) + C & \cong & A + (B + C) \\
(A +^* B) +^* C & \cong & A +^* (B +^* C).
\end{array}
\]
Moreover, these operations are sum-like, in the sense that for any orders $A$ and $B$ the sums $A + B$ and $A +^* B$ can be decomposed into two suborders, one of which is isomorphic to $A$ and the other to $B$. In the following sections, we will study other associative and sum-like binary operations $\oplus$ on $LO$ and subclasses of $LO$. 

Neither of the standard sums is commutative: there are many examples of orders $A$ and $B$ for which $A + B \not\cong B + A$, and hence $B +^* A \not\cong A +^* B$ (i.e., the operations $+$ and $+^*$ do not coincide up to isomorphism). We will also investigate the existence of associative and commutative sum-like operations on subclasses of $LO$. 

\section{Sums of linear orders}\label{sect: sums of linear orders}

In this section, we define the notion of a \textit{sum} of linear orders (\Cref{def: sums}). Before giving the definition, we provide some categorical motivation for it. We assume basic familiarity with category-theoretic terminology, though this will only be needed for the motivating discussion of this section, and will not be required for most of our later results. 

Let $\mathcal C_{LO}$ denote the category whose objects are the linear orders, with morphisms being the strictly order-preserving maps.

\begin{remark} \label{rmk: all morphs injective}
    By trichotomy and the order-preserving condition, all morphisms in this category are injective and are isomorphisms if they are onto.
\end{remark}

Attempting to uncover properties of this category, a natural question to ask is whether any of the standard categorical constructs manifest. For example, we could ask whether there is a natural notion of ``adding'' two objects as given by the coproduct in other categories. Regretfully, this is not the case.

\begin{theorem} \label{th: no coproducts}
    No pair of non-empty objects in $\mathcal{C}_{LO}$ has coproducts.
\end{theorem}

To prove this theorem we rely on the following lemma. Let $A$, $B$ be two non-empty linear orders.

\begin{lemma} \label{lm: coproduct is isomorphic to the usual sum}
    Any coproduct of $A$ and $B$ is isomorphic to $A+B$ (with its associated embeddings).
\end{lemma}

\begin{proof}
    Given any two linear orders $A$, $B$, say $A \oplus B$ is a coproduct of the two and $A \xhookrightarrow{f_1} A \oplus B$, $A \xhookrightarrow{f_2} A \oplus B$ are the associated morphisms. Since $A, B$ have morphisms into their usual sum $A + B$, by the universal property there exists a unique morphism $A \oplus B \xhookrightarrow{\phi} A + B$ such that the following diagram commutes:

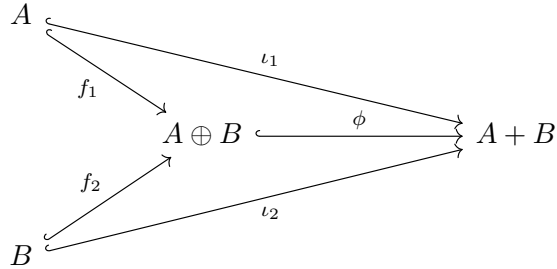
\[
\begin{tikzcd}[column sep=large, row sep=large]
A \arrow[dr, hook, "f_1"'] \arrow[drrr, hook, "\iota_1"] & & & & \\
& A \oplus B \arrow[rr, hook, "\phi"] & & A + B & \\
B \arrow[ur, hook, "f_2"] \arrow[urrr, hook, "\iota_2"'] & & & & 
\end{tikzcd}
\]

Where $\iota_1, \iota_2$ are the usual embeddings.

Note that $\phi$ must be an isomorphism since for any $x \in A + B$, either $x \in \iota_1(A)$, so $\phi(f_1(\iota_1^{-1}(x)))= x$; or $x \in \iota_2(B)$, and $\phi(f_2(\iota_2^{-1}(x)))= x$. 
\end{proof}

\begin{proof}[Proof of \ref{th: no coproducts}]

With $\iota_1$, $\iota_2$, $\tau_1$, $\tau_2$ being the usual inclusion mappings, by Lemma \ref{lm: coproduct is isomorphic to the usual sum} there must be a morphism $\phi$ such that the following diagram commutes 

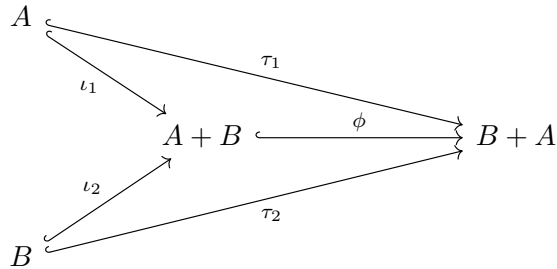
\[
\begin{tikzcd}[column sep=large, row sep=large]
A \arrow[dr, hook, "\iota_1"'] \arrow[drrr, hook, "\tau_1"] & & & & \\
& A + B \arrow[rr, hook, "\phi"] & & B + A & \\
B \arrow[ur, hook, "\iota_2"] \arrow[urrr, hook, "\tau_2"'] & & & & 
\end{tikzcd}
\]

For any $ a\in A$, $b \in B$, we must have $$\iota_1(a) < \iota_2(b)$$ therefore $$\tau(a) = \phi(\iota_1(a)) < \phi(\iota_2(b)) = \tau_2(b)$$ which is false.
\end{proof}

This theorem may also be proven by noticing that the proof of Lemma \ref{lm: coproduct is isomorphic to the usual sum} can be carried out verbatim by using the reverse sum ($(A, B) \mapsto B + A$), instead of the usual sum. This yields a contradiction since the usual sum is not commutative in general. 

Despite showing that the traditional categorical notion of sum is not present in $\mathcal{C}_{LO}$, Lemma \ref{lm: coproduct is isomorphic to the usual sum} points towards a viable generalization. 

Consider all operations that may be substituted in for the usual sum in the proof of Lemma \ref{lm: coproduct is isomorphic to the usual sum}. These will be operations $\oplus$ such that for any linear orders $A$ and $B$, we have associated embeddings $$A \lhook\joinrel\xhookrightarrow{\iota_1} A\oplus B,\ B \lhook\joinrel\xhookrightarrow{\iota_2} A\oplus B,$$
whose images encompass all of $A \oplus B$.
If we further assume that the images of these embeddings are disjoint\footnote{This condition can be done without, leading to the definition of a \textit{collapsed sum}. We will not consider such operations in this paper.}, we get the following alternative definition of a sum.

\begin{definition} \label{def: sums}
    A \textit{sum} $\oplus$ on $LO$ is a rule that given two linear orders $A, B$ assigns a new linear order $A \oplus B$ with associated morphisms $$A \lhook\joinrel\xhookrightarrow{e_A^{A\oplus B}} A\oplus B, B \lhook\joinrel\xhookrightarrow{e_B^{A\oplus B}} A\oplus B,$$ so that given any morphism $1 \xhookrightarrow{p} A\oplus B$, there is exactly one morphism $1 \xhookrightarrow{p_A} A$ or $1 \xhookrightarrow{p_B} B$, such that respectively either

    \begin{center}
        \begin{tikzcd}
1 \arrow[r, "p_A"] \arrow[rr, "p"', bend right] & A \arrow[r, "e_A^{A\oplus B}"] & A\oplus B
\end{tikzcd}
    \end{center}

or

\begin{center}
    \begin{tikzcd}
1 \arrow[r, "p_B"] \arrow[rr, "p"', bend right] & B \arrow[r, "e_B^{A\oplus B}"] & A\oplus B
\end{tikzcd}

\end{center}

commutes.
\end{definition}

\begin{remark}
    In order to remove any reference to non-categorical notions in this definition, we could carry out this definition in the category of non-strict linear orders with non-decreasing maps as morphism. In that case, the point order $1$ can be identified as the final object in the category.
\end{remark}

\begin{remark} \label{rmk: data associated to sums}
    Part of the data of a sum is the associated morphisms $e_A^{A\oplus B}$, $e_B^{A\oplus B}$, just like how part of the data required to define a coproduct is its associated morphisms for any given instantiation of it. One could say two sums $\oplus_1$, $\oplus_2$ are \textit{isomorphic for} $A$, $B$ just in case $(A\oplus_1B, e_A^{A\oplus_1B}, e_B^{A\oplus_1B})$ and $(A\oplus_2B, e_A^{A\oplus_2B}, e_B^{A\oplus_2B})$ are isomorphic in the category whose objects are diagrams of the form:

    \[\begin{tikzcd}
A \arrow["f_A^C", rd, hook] &   & B \arrow["f_B^C"', ld, hook] \\
                   & C &                   
\end{tikzcd}\]

    With morphisms between two such diagrams with sinks at $C_1$ and $C_2$ being morphisms $\varphi \in \text{Hom}_{\mathcal{C}_{LO}}(C_1, C_2)$ making the following diagram commute:

    \[
\begin{tikzcd}[column sep=large, row sep=large]
A \arrow["f_A^{C_1}"', dr, hook] \arrow["f_A^{C_2}", drrr, hook] & & & & \\
& C_1 \arrow[rr, hook, "\varphi"] & & C_2 & \\
B \arrow["f_B^{C_1}", ur, hook] \arrow["f_B^{C_2}"', urrr, hook] & & & & 
\end{tikzcd}
\]

Formally, we say that $\oplus_1$ and $\oplus_2$ are \textit{isomorphic} if they are isomorphic for all pairs of linear orders. Unless otherwise stated though, we will call two sums isomorphic if their base orders always are, without considering the associated embeddings. 
\end{remark}

We now provide a concrete characterization of sums\footnote{An alternate characterization using colorings will be given at the end of this section.}.

\begin{proposition} \label{prop: concrete characterization of sums}
     Any sum $\oplus$ is isomorphic to a sum whose base set is the disjoint union of the orders being summed.
\end{proposition}

\begin{proof}
    Let $(A, <_A)$, $(B, <_B)$ be two linear orders.
    
    Define an order $<^*$ on $A \sqcup B$ as follows, $c <^* d$ if and only if one of the following holds:
    
    \begin{enumerate}
        \item[(i)] $c, d \in A$ and $c <_A d$;
        \item[(ii)] $c, d \in B$ and $c <_B d$;
        \item[(iii)] $c \in A$, $d \in B$, $e_A^{A\oplus B} (c) < e_B^{A \oplus B} (d)$;
        \item [(iv)] $d \in A$, $c \in B$, $e_B^{A\oplus B} (c) < e_A^{A \oplus B} (d)$.
    \end{enumerate}

    With the embeddings associated to this sum being the inclusion maps to the disjoint union.

    Let $A\sqcup B \xhookrightarrow{\varphi} A\oplus B$ be such that for any $c \in A\sqcup B$ if $c \in A$, $\varphi(c) = e_A^{A\oplus B}(c)$, and if $c \in B$, $\varphi(c) = e_B^{A\oplus B}(c)$. Given any $c \in A\oplus B$, let $1 \xhookrightarrow{p} A\oplus B$ be the map that selects $c$. By definition of the sum, we may assume without loss of generality that there is only one morphism $1 \xhookrightarrow{p_A} A$ such that $e_A^{A\oplus B} (p_A (0)) = p(0) = c$, then $\varphi(p_A(0)) = c$. Hence, $\varphi$ is onto.

    To show that $(A\sqcup B, <^*)$ is a linear order, we check the axioms: clearly $<^*$ is irreflexive since $<_A$, $<_B$ are. If $a <^* b$, $b<^* c$, then either $a, b, c$ are all in either $A$ or $B$, in which case we're done, or we have six other cases to check, these are all analogous so we will only do one: if $a, b \in A$, $c \in B$, then $e_A^{A\oplus B} (a) < e_A^{A\oplus B} (b)$ (since $e_A^{A\oplus B}$ is a morphism) and $e_A^{A\oplus B}(b) < e_B^{A\oplus B}(c)$, so by transitivity of $<$, $e_A^{A\oplus B}(a) < e_B^{A\oplus B} (c)$, therefore $a < ^* c$. 
    
    Finally, for totality, say $b \not\leq^* a$, if $a, b$ are in the same base order we are immediately done. Without loss of generality, assume $a \in A$, $b \in B$, then $e_A^{A\oplus B} (a) \not<^* e_B^{A \oplus B}(b)$, we need to show that $e_A^{A\oplus B} (a) \neq e_B^{A \oplus B}(b)$. Let $1 \xhookrightarrow{p} A\oplus B$ be the morphism such that $p(0) = e_A^{A\oplus B}(a)$, since the morphism $1 \xhookrightarrow{p_A} A$ such that $p_A(0) = a$ satisfies $p = e_A^{A\oplus B} \circ p_A$, we can't have $e_A^{A\oplus B} (a) = e_B^{A \oplus B}(b)$, or else $1 \xhookrightarrow{p_B} B$ such that $p_B(0) = b$ would also satisfy $p = e_B^{A\oplus B} \circ p_B$, but there is exactly one morphism that will satisfy one of these conditions. Since we have shown that $e_A^{A\oplus B} (a) \not\leq^* e_B^{A \oplus B}(b)$, it follows that $e_B^{A\oplus B} (b) <^* e_A^{A \oplus B}(a)$, therefore $b <^* a$.

    That $\varphi$ is a morphism is immediate since $a <_A b$ and $c <_B d$ imply that $e_A^{A\oplus B} (a) < e_A^{A \oplus B} (b)$ and $e_B^{A\oplus B} (c) < e_B^{A \oplus B} (d)$ (respectively). It is onto, therefore making it an isomorphism. By the definition of this isomorphism these sums are isomorphic for $A$ and $B$.
\end{proof}

It remains to show that this definition of sum is not vacuous, and it is quite straightforward to see that the \textit{usual} and \textit{reverse} sums ($+$ and $+^*$) satisfy it. 

There are other examples of sums, many of which are quite silly.

\begin{example} \label{ex: non-standard, irregular, non-associative sum}
    Consider the operation
    
    \[
    A\oplus B = 
    \begin{cases}
        B + A, \text{ if } A = \{0\}\\
        A + B, \text{ else }
    \end{cases}
    \]

    This is an example of a \textit{non-standard} sum (i.e. a sum not isomorphic to the usual nor reverse sum).
\end{example}

In order to consider more interesting sums, we isolate some reasonable properties of sums by abstracting those of $+$ and $+^*$.

\subsection{Some reasonable properties of sums} \label{subsect: some reasonable properties of sums}

Two of the most reasonable properties of the standard sums are associativity and isomorphism-invariance.

\begin{definition} \label{def: associative}
    A sum $\oplus$ is \textit{associative} if, given orders $A$, $B$, $C$, we have

    $$A \oplus(B\oplus C) \cong (A \oplus B) \oplus C$$
\end{definition}

\begin{definition} \label{def: regular}
    A sum $\oplus$ is \textit{regular} if given linear orders $A$, $A'$, $B$, $B'$ such that $A\cong A'$, $B\cong B'$, then $A \oplus B \cong A' \oplus B'$.
\end{definition}

As mentioned above, the standard sums are both regular and associative.

\begin{example}
    The sum in Example \ref{ex: non-standard, irregular, non-associative sum} is neither associative, nor regular.
    
    We have: 
    $$\{1\} \oplus \mathbb{Q} = \{1\} + \mathbb{Q} \not\cong \mathbb{Q} + \{0\} = \{0\} \oplus \mathbb{Q} \text{ (not regular)};$$
    $$\{0\} \oplus (\{0\} \oplus \mathbb{Q}) = \{0\} \oplus (\mathbb{Q} + \{0\})= \mathbb{Q} + \{0\} + \{0\} \not\cong \{0\} + \{0\} + \mathbb{Q} = (\{0\} \oplus \{0\}) \oplus \mathbb{Q} \text{ (not associative).}$$

    It's dizzying to think about sums which aren't regular, given that the usual sum is itself defined up to isomorphism! If you don't agree, try to define an associative sum which is not regular.
\end{example}

\begin{example}
    We may tweak the previous sum to make a regular sum that is not associative. Consider

    \[
    A\oplus B = 
    \begin{cases}
        B + A, \text{ if } A \cong 1\\
        A + B, \text{ else }
    \end{cases}
    \]
\end{example}

\begin{example} \label{ex: more reasonable lead-up to +_w}
    A more reasonable example of a regular sum that is not associative is

    \[
    A\oplus B = 
    \begin{cases}
        n + A + B', \text{ if } B \cong n + B' \text{ for some } B' \text{ with no least element and } n \in \omega\\
        A + B, \text{ else }
    \end{cases}
    \]

    This fails to be associative since, for example, with $A = \mathbb{Q}$,
    $$(A \oplus \omega) \oplus n \cong n + (A + \omega) \not\cong A + \omega \cong A \oplus (\omega \oplus n).$$

    The problem here is that $\omega$ absorbs $n$.
\end{example}

The initial goal of this paper was to prove a characterization for the standard sums, without referencing their specific structure. Attempting to achieve this, we came up with a strengthening of associativity called \textit{tracking points}.

\begin{definition} \label{def: point-tracking}
    A sum $\oplus$ is said to \textit{track points} if, given any three linear orders $A$, $B$, $C$, there is an induced isomorphism such that the following diagram commutes:

\[\begin{tikzcd}
A \arrow[rd,  "e_A^{A\oplus B}", hook] 
  \arrow[rdd, "e_A^{A\oplus(B\oplus C)}"', hook] 
  & 
  & 
B \arrow[ld,  "e_B^{A\oplus B}"', hook] 
  \arrow[rd,  "e_B^{B\oplus C}", hook] 
  & 
  & 
C \arrow[ld,  "e_C^{B\oplus C}"', hook] 
  \arrow[ldd, "e_C^{(A\oplus B)\oplus C}", hook] 
  \\
& 
A\oplus B 
  \arrow[rrd, "e_{A\oplus B}^{(A\oplus B)\oplus C}", pos=0.1, hook] 
  & 
  & 
B\oplus C 
  \arrow[lld, "e_{B\oplus C}^{A\oplus (B\oplus C)}"', pos=0.1, hook] 
  & 
  \\
& 
A\oplus(B\oplus C) 
  \arrow[rr, "\cong", two heads, hook] 
  & 
  & 
(A\oplus B)\oplus C 
  & 
\end{tikzcd}
\]
\end{definition}

The usual and reverse sum are both examples of sums that track points. Sums which track points are associative in a very strong sense.

\begin{proposition} \label{prop: point-tracking characterization of standard sums}
    If $\oplus$ is a regular, associative sum that tracks points and, for all linear orders $A$, $A \oplus 1 \cong A + 1$ and $1 \oplus A \cong 1 + A$, then $\oplus$ is isomorphic to the usual sum. The analogous characterization for the reverse sum holds if $A \oplus 1 \cong 1 + A$ and $1 \oplus A \cong A + 1$.
\end{proposition}

\begin{proof}
    Consider $A, B \in LO$, if $B$ has a left endpoint, then $B \cong 1+ B'$, so
    $$A \oplus B \cong A \oplus (1 + B') \cong A \oplus (1 \oplus B') \cong (A \oplus 1) \oplus B' \cong (A + 1) \oplus B' \cong A + B,$$

    where the last step followed by point tracking since all elements of $B'$ must lie after the copy of $0 \in 1$ in the order $A \oplus B$ (this requires a small diagram chase).

    On the other hand, if $B$ has no left endpoint, for any $b \in B$, if $B_{<b} = \{x \in B \mid x<b\}$, $B_{\geq b} = \{x \in B \mid x\geq b\}$, we have
    
    $$A \oplus B = A \oplus (B_{<b} + B_{\geq b}) \cong A \oplus (B_{<b} \oplus B_{\geq b}) \cong (A \oplus B_{<b}) \oplus B_{\geq b} \cong (A \oplus B_{<b}) + B_{\geq b}$$

    So $b$ lies after all elements of $A$ in $A \oplus B$ but, since $b$ was arbitrary, this means that all elements of $e_B^{A\oplus B}(B)$ lie after $e_A^{A\oplus B}(A)$ in $A\oplus B$, hence the embedding condition implies that $A \oplus B \cong A +B$.
\end{proof}

Having proven this theorem, we were hopeful that a characterization of the standard sums as the only regular, associative sums would be forthcoming. Point-tracking appears to be a pervasive property of regular, associative sums, and so does the condition on point placement.

Neither of these assumptions hold in general, we will revisit point-tracking in \ref{sect: revisiting the properties of sums}. For now we'll attempt to break the latter hypothesis. Without deviating too much from familiar structure, say we want our sum to be as close to the usual sum as possible but with $A \oplus 1 \cong 1+A$ where does that lead us?

\begin{example} \label{ex: definition of +_w}
     Define $\oplus_0$ as
     
     \[
    A\oplus_0 B = 
    \begin{cases}
        1 + A + B', \text{ if } B \cong 1 + B' \text{ for some } B'\\
        A + B, \text{ else }
    \end{cases}
    \]
    
    This sum is regular, but not associative since, for example,
    
    $$(\mathbb{Q} \oplus_0 \omega) \oplus_0 1 \cong 2 + \mathbb{Q} + \omega \not\cong 1 + \mathbb{Q} + \omega\cong \mathbb{Q} \oplus_0 (\omega \oplus_0 1)$$

    Just like in Example \ref{ex: more reasonable lead-up to +_w}, the problem here is that $\omega$ can absorb a point during associativity. We will attempt to mitigate this going further and treating $\omega$ the same way we would treat a finite ordinal in Example \ref{ex: more reasonable lead-up to +_w}.

    Let

    \[
    A\oplus_1 B = 
    \begin{cases}
        n + A + B', \text{ if } B \cong n + B' \text{ for some } B' \text{ with no least element and } n \in \omega\\
        \omega + A + B', \text{ if } B \cong \omega + B' \text{ for some } B' \text{ with no least element}\\
        A + B, \text{ else }
    \end{cases}
    \]

    But again, we get a failure of associativity by associating with a bigger ordinal which in turn absorbs $\omega$ (and any $n$):
    $$(\mathbb{Q} \oplus_1 \omega\cdot\omega) \oplus_1 \omega = \omega + \mathbb{Q} + \omega \cdot \omega \not\cong \mathbb{Q} + \omega \cdot \omega \cong \mathbb{Q} \oplus_1 (\omega \cdot \omega \oplus_1 \omega)$$

    Noticing the pattern here (no matter how big an ordinal $\tau$ we shift to the left, we may always get a counterexample of associativity by getting a bigger ordinal, like $\tau \cdot \omega$, that absorbs it), we let

    \[
    A\oplus_2 B = 
    \begin{cases}
        \tau + A + B', \text{ if } B \cong \tau + B' \text{ for some } B' \text{ with no least element and } \tau \in Ord\\
        A + B, \text{ else }
    \end{cases}
    \]

    Note that the condition that $B'$ has no least element is necessary to remove ambiguity.

    It is straightforward to check that this non-standard sum is regular and associative using the fact that final segments of ordinals are still ordinals, and that the ordinals are closed under usual sums indexed by ordinals. Note that $\oplus_2$ also tracks points.

    From now on, we denote $\oplus_2$ as $+_w$.
\end{example}

We continue abstracting away reasonable properties of the standard sums, now strengthening regularity.

\begin{definition} \label{def: canonical regularity}
    A sum $\oplus$ is \textit{canonically regular} if given orders $A, A', B, B'$ with isomorphisms $\phi:A\xrightarrow{\sim}A'$, $\psi: B \xrightarrow{\sim} B'$, there is an induced isomorphism making the following diagram commute:

    \[\begin{tikzcd}
A \arrow[dd, "e_{A}^{A\oplus B}", hook] \arrow[rr, "\phi", two heads, hook] 
  &                                                                             & A' \arrow[dd, "e_{A'}^{A'\oplus B'}", hook, pos = 0.3] \\
  & B \arrow[ld, "e_{B}^{A\oplus B}", hook] \arrow[rr, "\psi", two heads, hook, pos = 0.4] 
  &                                                                             & B' \arrow[ld, "e_{B'}^{A'\oplus B'}", hook] \\
A\oplus B \arrow[rr, "\cong", two heads, hook] 
  &                                                                             & A' \oplus B'                                                                    
\end{tikzcd}\]
\end{definition}

The usual and reverse sums are both canonically regular, and in fact, so is $+_w$, since well-orders are preserved under isomorphism.

Examples of sums that are not canonically regular are more dynamic.

\begin{example} \label{ex: regular sum which is not canonically regular} 
    Consider

    \[
    A\oplus B = 
    \begin{cases}
        (-\infty, 0] + B + (0, \infty), \text{ if } A \cong \mathbb{R}\\
        A + B, \text{ else }
    \end{cases}
    \]

    This sum is regular, but not canonically regular. Note that for any $B$, $\varphi: \mathbb{R} \xrightarrow{\sim} \mathbb{R}$ such that $x \mapsto x + 1$ is an order automorphism, and the identity on $B$ is as well. That said, the induced map $\gamma: \mathbb{R} \oplus B \rightarrow \mathbb{R} \oplus B$ is not even a morphism, since for any $b \in B$, $0 < b$ in $\mathbb{R} \oplus B$, but $\gamma(0) = 1 > \gamma(b) = b$.
\end{example}

In a way, sums which are not canonically regular have to shuffle points in some obtrusive way. It is natural to ask whether there are any regular, associative sums which are not canonical. This will be answered in the positive in \Cref{sect: revisiting the properties of sums}. 

To conclude our discussion of basic properties, we codify how close a sum is to being standard.

\begin{definition} \label{def: standard on each side}
    Let $\oplus$ be a sum, a linear order $X$ is
    \begin{itemize}
        \item[(i)] \textit{right-}, \textit{left-standard} (with respect to $\oplus$) if for every order $A$, $X \oplus A \cong X + A$, $A \oplus X \cong A + X$ (respectively);
        \item[(ii)] \textit{right*-}, \textit{left*-standard} (with respect to $\oplus$) if for every order $A$, $X \oplus A \cong A + X$, $A \oplus X \cong X + A$ (respectively);
        \item[(iii)] \textit{standard} if it is always either right-standard and left-standard, or right*-standard and left*-standard.
    \end{itemize}
\end{definition}

\begin{proposition} \label{prop: 1 being standard of any sort implies the finite orders are too}
    If $\oplus$ is regular and associative and $1$ is (right-, left-, right*-, left*-) standard, then every finite order is (right-, left-, right*-, left*-) standard . 
\end{proposition}

\begin{proof}
    By Proposition \ref{prop: concrete characterization of sums} and looking at cardinalities, all sums agree with the usual sum on finite orders up to isomorphism.

    We will prove that all finite orders are right-standard, the proofs of the other cases are entirely analogous. Let $n$ be least such that $A \oplus n \not\cong A + n$ for some $A$. If $n > 1$, then $A \oplus n \cong A\oplus (n-1 \oplus \{n-1\}) \cong (A \oplus (n-1)) \oplus \{n-1\} \cong (A + (n-1)) \oplus \{n-1\} \cong A + n$. But this is a contradiction to the minimality of $n$, therefore $n$ must be to $1$ but, by assumption, $1$ is right-standard.
\end{proof}

\subsection{Semi-standard sums} \label{subsect: semi-standard sums}

\begin{definition} \label{def: semi-standard}
    Say a sum $\oplus$ is \textit{semi-standard} if given any orders $A, B$, their images $e_A^{A \oplus B}(A)$ and $e_B^{A \oplus B}(B)$ are convex in $A \oplus B$.
\end{definition}

\begin{remark} \label{rmk: equivalent condition to semi-standard}
    Equivalently, $\oplus$ is semi-standard if given any linear orders $A, B$, either $A \oplus B \cong A+B$ or $A\oplus B \cong B+A$.
\end{remark}

We will characterize the regular, associative, semi-standard sums as the standard sums. To achieve this, we introduce an essential result -- the distinguishing lemma. The proof of this lemma makes use of an important theorem due to Lindenbaum \cite{lindenbaum1926communication}.

\begin{theorem}[Fundamental Theorem] \label{th: fundamental theorem}
    If $X, Y$ are linear orders, $X$ is isomorphic to an initial segment of $Y$, and $Y$ is isomorphic to a final segment of $X$, then $X \cong Y$.
\end{theorem}

This theorem can be seen as a version of the Schr\"oder–Bernstein theorem applied to linear orders. In fact, the bijection between $X$ and $Y$ furnished by the proof of the Schr\"oder–Bernstein theorem will actually be an isomorphism of linear orders.

\begin{lemma}[Distinguishing] \label{lm: distinguishing}
    Given any set of linear orders $\mathfrak{C}$, there exists a linear order $X_\mathfrak{C}$ such that for any (possibly empty) finite sums of elements of $\mathfrak{C}$, $A$, $B$, $C$, $D$, and any isomorphism $f: A+X_\mathfrak{C}+B \xrightarrow{\sim} C+X_\mathfrak{C}+D$, $f(X_\mathfrak{C}) = X_\mathfrak{C}$, $f(A) = C$, $f(B) = D$. In fact, $X_\mathfrak{C}$ only depends on the supremum of the cardinality of the elements of $\mathfrak{C}$.

    Further, any non-empty, strict initial segment $I$ of $X_\mathfrak{C}$ can be extended to a longer initial segment isomorphic to it.
\end{lemma}

\begin{proof}
    Say $\kappa = \mbox{max}\{\aleph_0, \sup_{L \in \mathfrak{C}} |L|\}$. Let $X_\mathfrak{C}$ be the suborder of $(2^\kappa, <_{\text{lex}})$ consisting of the sequences which are not eventually constant. Note that $X_\mathfrak{C}$ has no endpoints and is $2^\kappa$-dense, since between any two points there is an isomorphic copy of itself.
    
    Say $A, B$ are finite sums of elements of $\mathfrak{C}$ and $f: A+X_\mathfrak{C}+B \xrightarrow{\sim} C+X_\mathfrak{C}+D$ is an isomorphism. Since $|A|, |B|, |C|, |D| \leq \kappa < 2^\kappa$, if $f(X_\mathfrak{C}) \cap D \neq \emptyset$, we get a contradiction from $|D| \geq |f(X_\mathfrak{C}) \cap D| = |f(X_\mathfrak{C})| = 2^\kappa$, since $f(X_\mathfrak{C}) \cap D$ has to be isomorphic to an interval of $X_\mathfrak{C}$; similarly if $f(X_\mathfrak{C}) \cap C \neq \emptyset$, therefore $f(X_\mathfrak{C}) = X_\mathfrak{C}$.

    Let $I \subseteq X_\mathfrak{C}$ be a non-empty, strict initial segment.

    If $r \in 2^{<\kappa}$, note that $I_r = \{\alpha \in X_\mathfrak{C} \; | \; \alpha \restriction \mbox{dom}(r) = r\} \cong X_\mathfrak{C}$, via the isomorphism $g: X_\mathfrak{C} \rightarrow I_r$ such that $g(\alpha) = s^\smallfrown\alpha$. If $L_r, R_r$ are the segments of points less than and greater than all points in $I_r$, then
    $$X_\mathfrak{C} \cong L_r +I_r+R_r.$$
    By the Fundamental Theorem, $I_r + R_r \cong X_\mathfrak{C} \cong L_r + I_r$.

    Since $(2^\kappa, <_{\mbox{lex}})$ is complete, let $\alpha \in 2^\kappa$ be the supremum of $I$. Let $r, s \in 2^{<\kappa}$ be such that $I_r< \alpha < I_s$, this is always possible since $\alpha \neq (0, 0, 0, \dots), (1, 1, 1, \dots)$ given that $I$ is non-empty and strict. If $M$ is the segment between $I_r$ and $I_s$,

    $$X_\mathfrak{C} \cong L_r + I_r + M + I_s + R_s,$$

    so by the above, $X \cong X + M + X$.

    Iterating this identity, we get that $$X \cong L + \dots + M_{-1} + X_{-1} + M_{0} + X_0 + M_1 + X_1 + \dots + R,$$ where for any $i \in \mathbb{Z}$, $M_i \cong M$ and $X_i \cong X$. Further, if $h: X \xrightarrow{\sim}  L + \dots + M_{-1} + X_{-1} + M_{0} + X_0 + M_1 + X_1 + \dots + R$ is such that the cut corresponding to $\alpha$ belongs to $M_0$. Consider the automorphism $$l: L + \dots + M_{-1} + X_{-1} + M_{0} + X_0 + M_1 + X_1 + \dots + R \xrightarrow{\sim} L + \dots + M_{-1} + X_{-1} + M_{0} + X_0 + M_1 + X_1 + \dots + R$$ such that $l$ is the identity on both $L$ and $R$, and for all $i \in \mathbb{Z}$, $l(M_i) = M_{i+1}$ and $l(X_i) = X_{i+1}$.

    Let $I' = (h^{-1}\circ l\circ h) (I)$. This is an initial segment isomorphic to $I$ which is longer than it.
\end{proof}

\begin{remark}
    In practice, the distinguishing lemma can be used like a global version of point-tracking in settings where sums shuffle unobtrusively.
\end{remark}

Any semi-standard sum with a small degree of consistency will in fact be fully standard. This result will characterize the standard sums as the regular, associative and semi-standard sums.

\begin{lemma} \label{lm: semi-standard sum with standard order is standard}
    Suppose $\oplus$ is a regular, associative, semi-standard sum with an order that is left-standard (left$^*$-standard), then $\oplus$ is the usual (reverse) sum.
\end{lemma}

\begin{proof}
    Say $A$ is left-standard and $B$, $C$ are arbitrary, consider $\mathfrak{C} = \mathcal{P}(A) \cup \mathcal{P}(B) \cup \mathcal{P}(C)$. If $W \in \{B, B \oplus C\}$,

    \[
    X_\mathfrak{C} \oplus(W \oplus A) \cong X_\mathfrak{C} \oplus (W + A) \cong 
    \begin{cases}
        X_\mathfrak{C} + W + A,\\
        W + A+ X_\mathfrak{C}
    \end{cases}
    \]

    this must be congruent to

    \[
    (X_\mathfrak{C} \oplus W) \oplus A\cong (X_\mathfrak{C} \oplus W) + A\cong 
    \begin{cases}
        X_\mathfrak{C} + W + A,\\
        W + X_\mathfrak{C} + A
    \end{cases}
    \]

    So, both of these things must equal $X_\mathfrak{C}+W+A$, thus, $X_\mathfrak{C} \oplus W \cong X_\mathfrak{C} + W$. Using this we show that $B \oplus C \cong B + C$.

    \[
    X_\mathfrak{C} \oplus(B \oplus C) \cong X_\mathfrak{C} + (B \oplus C) \cong 
    \begin{cases}
        X_\mathfrak{C} + B + C,\\
        X_\mathfrak{C} + C + B
    \end{cases}
    \]

    and,

    \[
    (X_\mathfrak{C} \oplus B) \oplus C \cong (X_\mathfrak{C} + B) \oplus C \cong 
    \begin{cases}
        X_\mathfrak{C} + B + C,\\
        C + X_\mathfrak{C} + B
    \end{cases}
    \]

    The only coherent case is the one in which both of these are isomorphic to $X_\mathfrak{C} + B + C$, therefore, we have that $B \oplus C \cong B + C$.
\end{proof}

Thus, in order to show that any regular, associative, semi-standard sum is isomorphic to the usual or reverse sum, we can proceed by contradiction, using the fact that there must be non-standard orders with respect to our sum. 

\begin{theorem} \label{th: semi-standard sums are standard}
    If $\oplus$ is a regular, associative, semi-standard sum, it is isomorphic to either the usual or reverse sum.
\end{theorem}

\begin{proof}
Let $X$ be a any non-empty order. Suppose toward a contradiction that $\oplus$ is not standard. By Lemma \ref{lm: semi-standard sum with standard order is standard}, there must be an order $A$ such that
\[
A \oplus X \cong A + X \not\cong X + A.
\]
There must also be an order $B$ such that
\[
X \oplus B \cong B + X \not\cong X + B.
\]

Let $\mathfrak{C} = \mathcal{P}(X) \cup \mathcal{P}(A) \cup \mathcal{P}(B)$. We claim that the following two identities hold:
\[
X_\mathfrak{C} \oplus A \cong X_\mathfrak{C} + A, \quad X_\mathfrak{C} \oplus (A + X) \cong X_\mathfrak{C} + A + X.
\]

Indeed, by associativity, we have
\[
X_\mathfrak{C} \oplus (A \oplus X) \cong (X_\mathfrak{C} \oplus A) \oplus X.
\]

Since $A \oplus X \cong A + X$, there are two possibilities for the sum on the left and four for the sum on the right. For the left, $X_\mathfrak{C} \oplus (A \oplus X)$ is either $X_\mathfrak{C} + A + X$ or $A + X + X_\mathfrak{C}$. In either case, $X_\mathfrak{C}$ appears on one side of the sum. This eliminates two of the possibilities on the right where $X_\mathfrak{C}$ appears in the middle. The remaining possibilities for the right sum are $X_\mathfrak{C} + A + X$ and $X + A + X_\mathfrak{C}$.

Now observe that of the four possible matches of the left sum to the right sum, the only non-contradictory one is that both are $X_\mathfrak{C} + A + X$ (we use that $A + X \not\cong X + A$ here). It follows that
\[
X_\mathfrak{C} \oplus (A + X) \cong X_\mathfrak{C} + A + X \quad \text{and} \quad X_\mathfrak{C} \oplus A \cong X_\mathfrak{C} + A,
\]
as claimed.

By a similar argument, we can show that
\[
B \oplus X_\mathfrak{C} \cong X_\mathfrak{C} + B, \quad (B + X) \oplus X_\mathfrak{C} \cong X_\mathfrak{C} + B + X.
\]

Now consider the sum
\[
X_\mathfrak{C} \oplus A \oplus X \oplus B \oplus X_\mathfrak{C}.
\]

The rest of the proof can be summarized as \textit{associate this sum in different ways until you reach a contradiction}. Here is one possible approach: consider,

\[
X_\mathfrak{C} \oplus (A \oplus X \oplus B) \oplus X_\mathfrak{C} \cong (X_\mathfrak{C} \oplus A \oplus X) \oplus (B \oplus X_\mathfrak{C}).
\]

For the sum on the left, we end up with a sum of one of the following forms:
\[
X_\mathfrak{C} + (\cdots) + X_\mathfrak{C}, \quad X_\mathfrak{C} + X_\mathfrak{C} + (\cdots), \quad (\cdots) + X_\mathfrak{C} + X_\mathfrak{C}.
\]

For the sum on the right, using our results above, we end up with a sum of the form:
\[
X_\mathfrak{C} + (\cdots) + X_\mathfrak{C} + (\cdots),
\]
which yields a contradiction in every case.

\end{proof}

Having achieved our initial goal of characterizing the standard sums, we now move on to investigating different ways of creating non-standard regular, associative sums, and their properties. Ultimately, we will have acquired enough tools to show that none of the reasonable properties from \Cref{subsect: some reasonable properties of sums} follow from regularity and associativity.

\subsection{Product bi-colorings} \label{subsect: product bi-colorings}
To conclude our introduction, we give an alternate concrete characterization of sums.

Given a sum $\oplus$, define a rule $P_\oplus$ that takes in two linear orders and outputs a 2-coloring on their product as follows, given $A, B\in LO$, let $P_\oplus(A, B) = p_{A \oplus B}: A \times B \rightarrow\{0, 1\}$ be such that:

\[p_{A\oplus B}(a, b) = 
    \begin{cases}
        1, \text{ if } e_A^{A \oplus B}(a) < e_B^{A \oplus B}(b)\\
        0, \text{ else}
    \end{cases}
\]

Call such a rule the \textit{associated product bi-coloring} of the sum. In general:

\begin{definition} \label{def: product bi-coloring}
    A \textit{product bi-coloring} is a rule $P$ that takes in two linear orders and outputs a 2-coloring on their product, such that given $A, B\in LO$, if $P(A, B) = p: A \times B \rightarrow\{0, 1\}$, then for all $a, a' \in A$, $b, b' \in B$

    \begin{enumerate}
        \item[(i)] if $b < b'$, then $p(a, b) \leq p(a, b')$;
        \item[(ii)] if $a' < a$, then $p(a, b) \leq p(a', b)$.
    \end{enumerate}
\end{definition}

\begin{proposition} \label{prop: product bi-colorings correspond to sums}
    Given any product bi-coloring $P$, there is a sum $\oplus_P$ whose associated product bi-coloring is $P$.
\end{proposition}

\begin{proof}
    Given two linear orders $A, B$, define $<_{\oplus_P}$ on the disjoint union $A \sqcup B$ by letting $<_{\oplus_P}$ respect the ordering on $A$ and $B$, and if $a \in A$, $b \in B$, let $a <_{\oplus_P} b$ if for $p = P(A, B)$, $p(a, b) = 1$, $b <_{\oplus_P} a$ else. Let the embeddings be the disjoint union embeddings.

    We prove this is a well-defined linear order. That it is a sum follows immediately since the ordering on the base orders is preserved.

    It is clearly irreflexive and connected. If three elements are in the same order, transitivity is immediate as well. Say $a, a' \in A$, $b, b' \in B$:

    If $a' <_{\oplus_P} a$, $a <_{\oplus_P} b$, then $1 = p(a, b) \leq p(a', b)$ by condition (ii), therefore, $p(a', b) = 1$ and $a' <_{\oplus_P} b$. The $b <_{\oplus_P} a' <_{\oplus_P} a$ case is completely analogous. If $a' <_{\oplus_P} b$, $b <_{\oplus_P} a$, then $p(a', b) = 1$ and $p(a, b) = 0$, so we can't have $a \leq a'$, therefore $a' <_{\oplus_P} a$.

    If $b <_{\oplus_P} a$, $a <_{\oplus_P} b'$, then $p(a, b) = 0$, $p(a, b') = 1$, so by condition (i), we can't have $b' \leq b$, therefore $b <_{\oplus_P} b'$. Similarly, if $b <_{\oplus_P} b'$, $b' <_{\oplus_P} a$, then $p(a, b') = 0$, hence $p(a, b) = 0$ as well, so $b <_{\oplus_P} a$.
\end{proof}

\section{Sum-generating classes} \label{sect: sum-generating classes}

In this section, we begin by looking at regular, associative sums which emulate the structure of $+_w$. In \Cref{th: characterization of sum-generating classes with simple sums}, we will characterize these sums as those resulting from so-called \textit{sum-generating classes}. In \Cref{subsect: filtrations and sifted sums}, we will show how to use filtrations of sum-generating classes to obtain sums with similar properties to the Hessenberg sum. These classes will be of independent interest due to their algebraic properties, which we will consider in \Cref{subsect: algebraic properties}.

\begin{example}
    Say a linear order $B$ is \textit{initially dense} if any initial segment of it embeds a dense suborder. The sum

    \[
    A +_s B = 
    \begin{cases}
        B_s + A + B', \text{ if } B \cong B_s + B' \text{ for some } B' \text{ initially dense and } B_s \text{ scattered}\\
        A + B, \text{ else }
    \end{cases}
    \]

    is regular and associative by the same argument that $+_w$ is.
\end{example}

Further, we could also separate initial segments from both summands and combine them in different ways.

\begin{example} \label{ex: form of simple sum}
    Consider the regular, associative sum $\oplus$ such that
    \[
    A \oplus B = \alpha + \beta + A' + B', \text{ if } A \cong \alpha + A', B \cong \beta + B' \text{ for } \alpha, \beta \text{ the longest well-ordered initial segments}.
    \]
\end{example}

Another variation on the theme of $+_w$ comes as follows. Let $\#$ be the Hessenberg sum on the well-orders\footnote{For a detailed introduction to this operation, refer to \Cref{subsect: good sums of well-orders}}.

\begin{example} \label{ex: initial global extension of the Hessenberg sum}
    Given linear orders $A$, $B$ so that $A, B \cong \gamma + A', \tau + B'$ for $A', B'$ with no least element and $\gamma, \tau \in Ord$,
       
    \[
    A +_h B = \tau \# \gamma + A' + B'
    \]

    is a regular, associative sum.
\end{example}

Classes like the well-orders and scattered orders, which yield sums by separating longest initial segments, are very versatile, allowing for a variety of different sums to be built from them. In this section, we will characterize such classes of orders.

We begin by characterizing classes of orders $\mathcal{C}$ that give rise to sums such as $+_s$ and $+_w$, that is, regular associative sums of the form

    \[
    A \oplus_\mathcal{C} B = \tag{1}
    \begin{cases}
        \beta + A + B', \text{ if } B \cong \beta + B' \text{ for the longest } \beta \in \mathcal{C}\\
        A + B, \text{ else }
    \end{cases}
    \]

\begin{remark}
    Note that part of what it means for a sum of this form to well-defined is that every linear order $B$ must have a (possibly empty) longest initial segment in $\mathcal{C}$.
\end{remark}

The following definition will provide our characterization.

\begin{definition} \label{def: left sum-generating class}
    A class $\mathcal{C}$ of linear orders is \textit{left sum-generating} (a \textit{left SGC}) if the following hold:
    
    \begin{enumerate}
        \item[(i)] $\mathcal{C}$ is closed under isomorphism;
        \item[(ii)] $A + B \in \mathcal{C}$ implies that $B \in \mathcal{C}$;
        \item[(iii)] if $A_\alpha \in \mathcal{C}$ for all $\alpha < \beta$, then $\sum_{\alpha < \beta} A_\alpha \in \mathcal{C}$.
    \end{enumerate}
\end{definition}

There is, of course, an analogous definition to the right.

\begin{definition} \label{def: right sum-generating class}
    A class $\mathcal{C}$ of linear orders is \textit{right sum-generating} (a \textit{right SGC}) if the following hold:
    
    \begin{enumerate}
        \item[(i)] $\mathcal{C}$ is closed under isomorphism;
        \item[(ii)] $A + B \in \mathcal{C}$ implies that $A \in \mathcal{C}$;
        \item[(iii)] if $A_\alpha \in \mathcal{C}$ for all $\alpha < \beta$, then $\sum_{\alpha < \beta^*} A_\alpha \in \mathcal{C}$.
    \end{enumerate}
\end{definition}

\begin{remark} \label{rmk: complementary sum-generating class}
    Given a left SGC $\mathcal{C}$, there exists a complementary right SGC $\mathcal{C}^\perp$ defined by $A \in \mathcal{C}^\perp$ iff no non-empty initial segment of $A$ is in $\mathcal{C}$. Likewise, we may define the complement operation on right SGCs by using final segments instead.

    For any SGC $\mathcal{C}$, the complement behaves as expected, $(\mathcal{C}^\perp)^\perp = \mathcal{C}$. Consider the case that $\mathcal{C}$ is left sum-generating, then $A \in (\mathcal{C}^\perp)^\perp$ iff $A$ has no non-empty final segment in $\mathcal{C}^\perp$ iff every non-empty final segment of $A$ has a non-empty initial segment in $\mathcal{C}$. Using closure under ordinal sums and final segments, we get that $A \in \mathcal{C}$ iff $A \in (\mathcal{C}^\perp)^\perp$. The case for right SGCs is entirely analogous.
\end{remark}

We could have also given an equivalent definition of SGCs which will be useful for the proof of the main characterization.

\begin{proposition} \label{prop: left-sum-generating classes are closed under right-limits}
    If $\mathcal{C}$ satisfies conditions (i) and (ii) from Definition \ref{def: left sum-generating class}, the following are equivalent:
    \begin{enumerate}
        \item $\mathcal{C}$ is a left SGC;
        \item \begin{enumerate}
            \item[(i)] for any $A, B \in \mathcal{C}$, $A+B \in \mathcal{C}$;
            \item[(ii)] if $A_0 \preceq A_1 \preceq \dots$ is any sequence of elements of $\mathcal{C}$ of arbitrary length such that for any of the embeddings $\varphi_\alpha^\beta: A_\alpha \xhookrightarrow{} A_\beta$, $\varphi_\alpha^\beta(A_\alpha)$ is an initial segment of $A_\beta$, then the direct limit of this sequence is in $\mathcal{C}$ (\textit{closure under right limits}).
        \end{enumerate}
    \end{enumerate}

    The analogous equivalence holds for right SGCs with limits of final segments (\textit{closure under left limits})\footnote{Throughout this section we will focus on left sum-generating classes, the proofs and constructions are entirely analogous for right SGCs.}.
\end{proposition}

\begin{proof}
    $(\Rightarrow)$ For 2.(i), consider a sequence of length $\beta = 2$ using (iii) in Definition \ref{def: left sum-generating class}. 2.(ii) follows by induction on the length of the sequence of elements. 
    
    The base case is immediate. If the sequence is of length $\alpha + 1$, then the limit is $A_{\alpha} \in \mathcal{C}$. If the sequence is of length $\lambda \in Lim-\{0\}$, for any $\alpha < \lambda$ define
    $$B_\alpha = A_\alpha - \bigcup_{\beta<\alpha}\varphi_\beta^\alpha(A_\beta),$$

    these orders are in $\mathcal{C}$ by condition (ii). Note that $$\lim_{\alpha < \lambda} A_\alpha = \sum_{\alpha < \lambda} B_\alpha.$$

    $(\Leftarrow)$ By induction. The base case is again immediate, and the successor step is covered by 2.(i) with the inductive hypothesis. If $\lambda \in Lim-{0}$, let the $A'_\alpha$ be the sequence of partial sums
    $$A'_\alpha = \sum_{\beta < \alpha} A_\beta$$
    with the embeddings just being inclusions. Take their limit, which will be the complete sum.
\end{proof}

Finally, we come to our characterization -- the theorem that gives SGCs their name.

\begin{theorem} \label{th: characterization of sum-generating classes with simple sums}
    The class $\mathcal{C}$ is left sum-generating iff $\oplus_\mathcal{C}$ is a regular, associative sum of the form $(1)$.
\end{theorem}

\begin{proof}
    ($\Rightarrow$) First we show that $\oplus_\mathcal{C}$ is well-defined.
    
    Given any linear order $B$, if $B$ has no initial segment in $\mathcal{C}$, we are done; otherwise, let $$I = \{\beta \in \mathcal{P}(B) \mid \beta \text{ is an initial segment of } B, \beta \in \mathcal{C}\},$$ then by condition 2.(ii), the union of $I$ is in $\mathcal{C}$ and is the largest initial segment of $B$ in $\mathcal{C}$. Formally, let $A_0 \in \beta$ be arbitrary and, inductively, given $A_\alpha$, let $A_\alpha \subseteq A_{\alpha+1} \in I$ (if such an element of $I$ exists). At a limit stage $\lambda$, let $A_\lambda = \bigcup_{\alpha < \lambda} A_\alpha$. This yields a sequence $A_0 \subseteq A_1 \subseteq \dots$ of the form in 2.(ii) whose direct limit is $\bigcup I$.  

    By condition (i), the sum is regular. This is easy to check, say $f: B \rightarrow C$ is an isomorphism, then if we decompose $B$ as $\beta + B'$ and $C$ as $\gamma + C'$, we claim $f(\beta) = \gamma$. Clearly $f(\beta) \in \mathcal{C}$ by condition (i), thus $f(\beta) \subseteq \gamma$, but $f^{-1}(\gamma) \in \mathcal{C}$ and is also an initial segment, so $f^{-1}(\gamma) \subseteq \beta$ therefore $\gamma \subseteq f(\beta)$, hence also $f(B') = C'$.

    For associativity, say $A \cong \alpha + A'$, $B \cong \beta + B'$, $C \cong \gamma + C'$:

    \[(A \oplus_\mathcal{C} B) \oplus_\mathcal{C} C \cong (\beta + \alpha + A' + B') \oplus_\mathcal{C} C \cong \gamma + \beta + \alpha + A' + B' + C' \tag{2}\]

    Also:

    \[A \oplus_\mathcal{C} (B \oplus_\mathcal{C} C) \cong A \oplus_\mathcal{C} (\gamma+\beta+B'+C') \tag{3}\]

    Now, we have to verify that the longest initial segment of $\gamma+\beta+B'+C'$ is indeed $\gamma + \beta$. $\gamma + \beta \in \mathcal{C}$, so the longest initial segment has to contain $\gamma + \beta$, say it is $\gamma + \beta + \theta_B + \theta_C$, where $\theta_B$, $\theta_C$ are segments of $B'$ and $C'$ (respectively) where both $\theta_B$ and $\theta_C$ are allowed to be the empty order. Note that $\theta_C$ must indeed be empty since by condition (ii), $(\gamma + \beta + \theta_B) + \theta_C \in \mathcal{C}$ implies that $\theta_C \in \mathcal{C}$. But we may apply condition (ii) again to $(\gamma + \beta) + \theta_B$, to get that $\theta_B \in \mathcal{C}$, so it must also empty.

    Therefore (2) $\cong$ (3). 

    ($\Leftarrow$) For this direction we'll need some way to track where each segment of $B$ goes in the sum. We use the distinguishing lemma.

    First we show (i) holds of $\mathcal{C}$. Say $\beta \in \mathcal{C}$, $\beta \cong \beta'$, and let $\mathfrak{C} = (\{\beta\} \cup \mathcal{P}(\beta')) - \{\emptyset\}$ (where subsets are ordered as suborders) to obtain $X_\mathfrak{C}$ as in Lemma \ref{lm: distinguishing}. By regularity, there is an isomorphism $f: X_\mathfrak{C} \oplus_\mathcal{C} \beta \xrightarrow{\sim} X_\mathfrak{C} \oplus_\mathcal{C} \beta'$, and since $X_\mathfrak{C} \oplus_\mathcal{C} \beta = \beta + X_\mathfrak{C}$, $X_\mathfrak{C} \oplus_\mathcal{C} \beta' \cong \varepsilon + X_\mathfrak{C} + \delta$, where $\varepsilon$ is the longest initial segment of $\beta'$ in $\mathcal{C}$ and $\beta' \cong \varepsilon + \delta$. By the distinguishing lemma, $f(\delta) = \emptyset$ and $f(\varepsilon) \cong\beta$, so $\beta' = \varepsilon$ and we must have that $\beta' \in \mathcal{C}$ in order for $X_\mathfrak{C} \oplus_\mathcal{C} \beta' = \beta' + X_\mathfrak{C}$.

    2.(ii) follows by the well-definedness of $\oplus_\mathcal{C}$. Say $A_0 \preceq A_1 \preceq \dots$ is a sequence of elements of $\mathcal{C}$ of the type described in 2.(ii), consider its direct limit $A$. Every $A_i$ is isomorphic to an initial segment of $A$, so there is a cofinal sequence of initial segments of $A$ which are in $\mathcal{C}$, thereby in order for $\oplus_\mathcal{C}$ to be well-defined, we must have that $A$ is itself in $\mathcal{C}$.

    2.(i) requires a little symbol shuffling.
    
    Say $\alpha, \beta \in \mathcal{C}$, let $\mathfrak{C} = (\mathcal{P}(\alpha) \cup \mathcal{P}(\beta)) - \{\emptyset\}$. 

    $$\alpha + \beta + X_\mathfrak{C} = (X_\mathfrak{C} \oplus_\mathcal{C} \beta) \oplus_\mathcal{C} \alpha \cong X_\mathfrak{C} \oplus_\mathcal{C} (\beta \oplus_\mathcal{C} \alpha) = X_\mathfrak{C} \oplus_\mathcal{C} (\alpha + \beta) \cong \varepsilon + X_\mathfrak{C} + \delta$$

    Where $\varepsilon + \delta \cong \alpha + \beta$ ($\varepsilon$ is the longest initial segment of $\alpha + \beta$ in $\mathcal{C}$). By the distinguishing lemma, $\delta = \emptyset$ and $\alpha + \beta \cong \epsilon \in \mathcal{C}$, so $\alpha + \beta \in \mathcal{C}$ by (i).

    For (ii), say $\alpha + \beta \in \mathcal{C}$. Let $\mathfrak{C} = (\mathcal{P}(\alpha) \cup \mathcal{P}(\beta)) - \{\emptyset\}$ we want to show that $\beta \in \mathcal{C}$. 
    
    Assume $\alpha \in \mathcal{C}$, then if $\beta \cong \Bar{\beta} + \beta'$, where $\Bar{\beta}$ is the longest initial segment of $\beta$ in $\mathcal{C}$,

    $$\alpha + \Bar{\beta} + X_\mathfrak{C} + \beta' \cong (X_\mathfrak{C} \oplus_\mathcal{C} \beta) \oplus_\mathcal{C} \alpha \cong X_\mathfrak{C} \oplus_\mathcal{C} (\beta \oplus_\mathcal{C} \alpha) = X_\mathfrak{C} \oplus_\mathcal{C} (\alpha + \beta) = \alpha + \beta + X_\mathfrak{C}$$

    So, $\beta' = \emptyset$ and $\beta = \Bar{\beta} \in \mathcal{C}$.

    If $\alpha \notin \mathcal{C}$, consider its longest initial segment in $\mathcal{C}$, say $\alpha \cong \Bar{\alpha} + \alpha'$. By the above $\alpha' + \beta \in \mathcal{C}$, so we can assume no initial segment of $\alpha$ is in $\mathcal{C}$. Say $\beta \cong \Bar{\beta} + \beta'$ is again any decomposition where $\Bar{\beta} \in \mathcal{C}$ but no initial segment of $\beta'$ is in $\mathcal{C}$, then

    $$\alpha + \beta + X_\mathfrak{C} \cong X_\mathfrak{C} \oplus_\mathcal{C} ((\alpha + \Bar{\beta}) \oplus_\mathcal{C} \beta') \cong (X_\mathfrak{C} \oplus_\mathcal{C} (\alpha + \Bar{\beta})) \oplus_\mathcal{C} \beta' \cong (X_\mathfrak{C} \oplus_\mathcal{C} (\alpha + \Bar{\beta})) + \beta'$$

    Tells us that since $\beta'$ will stay to the right of $X_\mathfrak{C}$, we must have $\beta' = \emptyset$, therefore $\beta = \Bar{\beta} \in \mathcal{C}$, and we are done.
\end{proof}

\begin{remark}
    An alternative sum to that given in (1) takes a right SGC $\mathcal{C}$, and separates the right summand's longest final segment. The above characterization holds analogously.

    Note that this sum would be the same as (1) with the left SGC $\mathcal{C}^\perp$.
\end{remark}

The sum in (1) is called the \textit{simple sum generated by $\mathcal{C}$}.

Simple sums allow us to generate a variety of different non-standard sums. If $L$ is any non-empty linear order, let $\mathcal{C} = \langle L \rangle$, be the left SGC generated by $L$. The simple sum generated by $\mathcal{C}$ will be a non-standard regular, associative sum.

The standard sums are also examples of simple sums. The usual sum is the simple sum generated by $\{\emptyset\}$, and the reverse sum is generated by $LO$.

Simple sums are extremely well-behaved, we will show that they are all canonically regular in \ref{subsect: algebraic properties}. They will also track points under the normal embeddings. This follows immediately from the definition of point-tracking, considering that $$A \oplus (B \oplus C) \cong C_L + B_L + A + B_R + C_R \cong (A \oplus B) \oplus C,$$ while $A\oplus B = B_L + A + B_R$ and $B \oplus C = C_L + B + C_R$. The induced isomorphism will be the composition of the two isomorphisms given above where each summand maps to their corresponding sections in the result.

As mentioned previously, SGCs are quite flexible, in fact, they allow us to characterize sums of the form in \Cref{ex: initial global extension of the Hessenberg sum}.

\begin{definition} \label{def: sum on a class}
    Given a class of orders $C$, we say $\oplus$ is a \textit{sum on $C$}, if it is well-defined within $C$ and can be extended to a (global) sum\footnote{Given any sum on $C$, we may extend it to a global sum by letting it behave exactly the same as the usual sum outside its domain. That said, this extension may not retain any of the restriction's properties.}. 
    
    Properties such as associativity and regularity carry over with their usual meanings in this restricted setting.
\end{definition}

Whenever we consider a regular sum $\oplus$ on $C$, we will assume that $C$ is closed under isomorphism, likewise, if $\oplus$ is associative on $C$, we will require that its outputs be in $C$.

\begin{proposition} \label{prop: sum extending sum on an SGC}
    If $\mathcal{C}$ is a left SGC and $\oplus'$ is a regular, associative sum on $\mathcal{C}$, the sum $\oplus$ on $LO$ defined by \[A \oplus B = A_L\oplus'B_L+A_R+B_R, \] where \( A_L \) and \( B_L \) are the longest initial segments of \( A \) and \( B \) that belong to \( \mathcal{C} \), is an associative regular sum.
\end{proposition}

The proof of this proposition is not difficult, especially considering the beginning of the proof of Theorem \ref{th: characterization of sum-generating classes with simple sums}. We also obtain a converse.

\begin{theorem} \label{th: shuffle sum of sum-generating segments}
Suppose that \( \mathcal{C} \) is a class of linear orders closed under isomorphism, such that

\begin{enumerate}
    \item every linear order \( A \) has a longest initial segment \( A_L \) in \(\mathcal{C} \), and
    \item the sum \( \oplus \) from Proposition \ref{prop: sum extending sum on an SGC} with $\oplus' = +$ is associative.
\end{enumerate}

Then the class \( \mathcal{C} \) is a left SGC.
\end{theorem}

\begin{proof}

We need to show that \( \mathcal{C} \) is closed under final segments, sums, and right limits. As mentioned in Theorem \ref{th: characterization of sum-generating classes with simple sums}, condition (1) implies that it is closed under right limits.

For closure under sums, fix \( \alpha, \beta \in \mathcal{C} \) and suppose toward a contradiction that \( \alpha+\beta \notin \mathcal{C} \). Let $\mathfrak{C} = \mathcal{P}(\alpha) \cup \mathcal{P}(\beta)$, and let $X_\mathfrak{C}$ be as in the distinguishing lemma.

If there is a non-empty initial segment $I$ of $X_\mathfrak{C}$ such that \( I \in \mathcal{C} \), then the final conclusion of the distinguishing lemma furnishes a find a cofinal sequence of initial segments of \( X_\mathfrak{C} \) that are isomorphic to \( I \). Since \( \mathcal{C} \) is closed under right limits, it follows \( X_\mathfrak{C} \in \mathcal{C} \). Thus, either \( X_\mathfrak{C} \in \mathcal{C} \) or no non-empty initial segment of \( X_\mathfrak{C} \) belongs to \( \mathcal{C} \). We work through these cases in turn.

First, suppose \( X_\mathfrak{C} \in \mathcal{C} \). Since \( \alpha, \beta \in \mathcal{C} \), we have \( \alpha \oplus \beta = \alpha + \beta \). Since \( \alpha+\beta \notin \mathcal{C} \) we have \( \alpha+\beta = A+B \), where \( A \) is the maximal initial segment of \( \alpha+ \beta \) belonging to \( \mathcal{C} \) and \( B \) is non-empty.

On one hand we have
\[
(\alpha \oplus \beta) \oplus X_\mathfrak{C} = (\alpha+ \beta) \oplus X_\mathfrak{C} = A + X_\mathfrak{C} + B,
\]
and on the other
\[
\alpha \oplus (\beta \oplus X_\mathfrak{C}) = \alpha \oplus (\beta + X_\mathfrak{C}) = \alpha + \beta + X_\mathfrak{C}.
\]

By associativity we have \( A + X_\mathfrak{C} + B \cong \alpha + \beta + X_\mathfrak{C} \), a contradiction by the distinguishing lemma.

Now, suppose no non-empty initial segment of \( X_\mathfrak{C} \) belongs to \( \mathcal{C} \). Observe that by definition of \( \oplus \) we have \( X_\mathfrak{C} \oplus \alpha = \alpha + X_\mathfrak{C} = \alpha \oplus X_\mathfrak{C} \). In one instance of associativity
\[
X_\mathfrak{C} \oplus (\alpha \oplus \beta) = X_\mathfrak{C} \oplus (\alpha + \beta) = A + X_\mathfrak{C} + B,
\]
and in the other
\[
(X_\mathfrak{C} \oplus \alpha) \oplus \beta = (\alpha \oplus X_\mathfrak{C}) \oplus \beta = \alpha \oplus (\beta + X_\mathfrak{C}) = \alpha + \beta + X_\mathfrak{C}.
\]

Again a contradiction by associativity. Thus \( \mathcal{C} \) is closed under sums.

For closure under final segments, suppose \( \alpha \in \mathcal{C} \) and \( \alpha \cong \Bar{\alpha} + \alpha' \). We wish to show \( \alpha' \in \mathcal{C} \). Suppose toward a contradiction this is false, so that \( \alpha' = A + B \) where \( A \) is the maximal initial segment of \( \alpha' \) in \( \mathcal{C} \) and \( B \) is non-empty. Since \( \mathcal{C} \) is closed under sums we must have that \( B \) has no non-empty initial segment \( B' \) belonging to \( \mathcal{C} \). It follows by definition of \( \oplus \) that \( (\Bar{\alpha} + A) \oplus B = \Bar{\alpha} + A + B \), i.e. \( \alpha \cong (\Bar{\alpha} + A) \oplus B \).

Let $\mathfrak{C} = \mathcal{P}(\alpha)$ and define \( X_\mathfrak{C} \) as above. Again we work through the two possible cases.

Suppose first \( X_\mathfrak{C} \in \mathcal{C} \). We have that
\[
((\Bar{\alpha} + A) \oplus B) \oplus X_\mathfrak{C} \cong \alpha \oplus X_\mathfrak{C} = \alpha + X_\mathfrak{C},
\]
and also
\[
(\Bar{\alpha} + A) \oplus (B \oplus X_\mathfrak{C}) = (\Bar{\alpha} + A) \oplus (X_\mathfrak{C} + B).
\]

Suppose \( X_\mathfrak{C} + B \) decomposes as \( C + D \), where \( C \) is the maximal initial segment of \( X_\mathfrak{C} + B \) in \( \mathcal{C} \). Since \( X_\mathfrak{C} \in \mathcal{C} \), \( X_\mathfrak{C} \) is an initial segment of \( C \). Thus regardless of the decomposition of \( (\Bar{\alpha} + A) \), \( X_\mathfrak{C} \) appears as a non-final segment in the sum \( (\Bar{\alpha} + A) \oplus (X_\mathfrak{C} + B) \). Since we must have \[ \alpha + X_\mathfrak{C} \cong (\Bar{\alpha} + A) \oplus (X_\mathfrak{C} + B), \] we get a contradiction as before.

Now suppose no initial segment of \( X_\mathfrak{C} \) belongs to \( \mathcal{C} \), then
\[
X_\mathfrak{C} \oplus ((\Bar{\alpha} + A) \oplus B) \cong X_\mathfrak{C} \oplus \alpha = \alpha + X_\mathfrak{C}.
\]

On the other hand, we have $$(X_\mathfrak{C} \oplus (\Bar{\alpha} + A)) \oplus B \cong (X_\mathfrak{C} \oplus (\Bar{\alpha} + A)) + B.$$
Thus, by associativity \( B \) appears as a non-empty final segment in $\alpha + X_\mathfrak{C}$, again a contradiction. 

Hence, \( \mathcal{C} \) is closed under final segments, and we are done.
\end{proof}

\subsection{Filtrations and sifted sums} \label{subsect: filtrations and sifted sums}

So far, we shown how to construct two of the sums given at the introduction to this section. How are we to make sense of $+_h$ -- the global extension of the Hessenberg sum?

We get a partial answer from Proposition \ref{prop: sum extending sum on an SGC} but, in general, how do we construct sums on sum-generating classes? Ideally, we would like to find some approach to defining sums like the Hessenberg sum ($\#$) by leveraging the structure of the SGC it is built on. To do this, we will use filtrations of sum-generating classes.

Let $\dots\subseteq\mathcal{C}_2\subseteq\mathcal{C}_1\subseteq\mathcal{C}_0$ be a decreasing filtration of left sum-generating classes of possibly unbounded ordinal length. Given some order $L$, let $L'_\alpha$ be the longest initial segment of $L$ in $\mathcal{C}_\alpha$ for all $\alpha \in Ord$. Let $L_\alpha = L_\alpha'-L_{\alpha+1}'$ and $\hat{\mathcal{C}}_\alpha = \{L_\alpha \;|\; L \in LO\}$.

Let\footnote{While adding over $Ord^*$ is not possible, for any give order $L$, there will be an $\alpha \in Ord$ such that for all $\beta \geq \alpha$, $L_\beta = \emptyset$. We take $\sum_{\alpha \in Ord^*} L_\alpha$ to be shorthand for $\sum_{\beta \in \alpha^*} L_\beta$.}
\[\tag{4}A \oplus B = \left ( \sum_{\alpha \in Ord^*}A_\alpha\oplus_\alpha B_\alpha\right) + A_R+B_R\]

where $A_R, B_R$ are the respective longest final segments in $\mathcal{C}_0^\perp$, and for all $\alpha$, the $\oplus_{\alpha}$ are sums.

\begin{example} \label{ex: global extension of the Hessenberg sum}
    If for all $\alpha$ we let $\oplus_\alpha = +$ and $\mathcal{C}_\alpha = \langle\omega^\alpha\rangle$, the resulting global sum will be $+_h$.
\end{example}

In general, we call the process of constructing a sum of this sort \textit{sifting}, and the sum in (4) is the \textit{sifted sum generated by} $\{(\mathcal{C}_\alpha, \oplus_\alpha)\}_{\alpha \in Ord}$. This method will be essential when constructing global sums that do not satisfy the reasonable properties of sums from \ref{subsect: some reasonable properties of sums}, so we will develop it more for that purpose.

Consider orders $A, B, C$, if $\oplus$ is as in (4),

$$(A \oplus B) \oplus C =  \left (\left ( \sum_{\alpha \in Ord^*} A_\alpha \oplus_\alpha B_\alpha\right) + A_R+B_R \right) \oplus C.$$

We would like to show this is isomorphic to

$$\left ( \sum_{\alpha \in Ord^*} \left ( \sum_{\alpha \in Ord^*}A_\alpha\oplus_\alpha B_\alpha\right)_\alpha \oplus_\alpha C_\alpha\right) + A_R+B_R + C_R,$$

in other words, that $\sum_{\alpha \in Ord^*}A_\alpha\oplus_\alpha B_\alpha \in \mathcal{C}$. In which case, to prove associativity, we will just need

$$(A \oplus B)_\beta = \left ( \sum_{\alpha \in Ord^*}A_\alpha\oplus_\alpha B_\alpha\right)_\beta = A_\beta \oplus_\beta B_\beta.$$

We call a sifting scheme $\{(\mathcal{C}_\alpha, \oplus_\alpha)\}_{\alpha \in Ord}$ whose associated sum satisfies all these properties, \textit{effective}. 

In section 4, we will use sifting as a way to isolate cardinality in SGCs, in particular, from section 3 we will obtain a method of constructing sums at specific cardinalities, and we will require some way to join all these different sums together. Given an SGC $\mathcal{C}$, we will consider the decreasing filtration $$\mathcal{C}^{card}_\alpha = \{L \in \mathcal{C} \mid L \text{ has no final segment of cardinality} <\aleph_\alpha\}.$$

Further, we will be working with SGCs with similar structure to $\mathcal{W}$. Consider an SGC $\mathcal{C}$ generated by a finite set $S = \{L_1, \dots, L_n\}$, which is closed under non-empty final segments (up to isomorphism) and whose elements all have the same cardinality\footnote{If one of the orders has a maximal element, then $\mathcal{C} = \mathcal{W}$.}. In this case, $\mathcal{C}$ will simply consist of all ordinal-indexed sums of orders in $S\cup\{\emptyset\}$.

Applying the cardinality filtration, note that $\hat{\mathcal{C}}^{card}_\alpha = \{\emptyset\}$ for all $\alpha$ such that $|L_1| > \aleph_\alpha$, and will subsequently consist of sums of orders in $S$ indexed by ordinals of cardinality $\aleph_\alpha$ with no final segments of smaller cardinality. In this later case, non-empty elements of $\hat{\mathcal{C}}^{card}_\alpha$ will have cardinality $\aleph_\alpha$.

Whether the scheme $\{(\mathcal{C}^{card}_\alpha, \oplus_\alpha)\}_{\alpha \in Ord}$ is effective depends on the choice of sums $\oplus_\alpha$.

\begin{proposition} \label{prop: cardinality sifting is effective}
    In the above setting, if the sums $\oplus_\alpha$ are sums on $\hat{\mathcal{C}}^{card}_\alpha$, then $\{(\mathcal{C}^{card}_\alpha, \oplus_\alpha)\}_{\alpha \in Ord}$ will be an effective filtration.
\end{proposition}

\begin{proof}
    Without loss of generality, we only need to consider sums of orders in $\mathcal{C}$. Say $A, B \in \mathcal{C}$, the sum $\sum_{\alpha \in Ord^*}A_\alpha\oplus_\alpha B_\alpha \in \mathcal{C}$ will still be in $\mathcal{C}$, since it is also an ordinal-indexed sum of elements in $S$.

    By induction on $\beta$, if $\aleph_\beta < |L_1|$, then

    $$(A \oplus B)_\beta = \emptyset = A_\beta \oplus_\beta B_\beta.$$

    If $|L_1| = \aleph_\beta$, then $(A \oplus B)_\beta$ will be the final segment of $A \oplus B$ indexed by an ordinal smaller than $\aleph_\beta$, given that each $\oplus_\alpha$ remains within each $\hat{\mathcal{C}}^{card}_\alpha$, by definition this will just be $A_\beta \oplus_\beta B_\beta$, since $A_\beta \oplus_\beta B_\beta \subseteq (A \oplus B)_\beta$, and equality holds since any final segment of $A_{\beta+1} \oplus_{\beta+1} B_{\beta+1}$ has at least cardinality $\aleph_\beta^+$, so it can not be an initial segment of an element of $\hat{\mathcal{C}}^{card}_\beta$.

    If this holds for all $\delta<\beta$, note that $A_\beta \oplus_\beta B_\beta$ has to be a final segment of $(A \oplus B)_\beta$, since it cannot be in $\mathcal{C}^{card}_{\beta+1}$ given that it has cardinality $\aleph_\beta$. But, since $A_{\beta+1} \oplus_{\beta+1} B_{\beta+1} \in \hat{\mathcal{C}}^{card}_{\beta+1}$, it must be all of $(A \oplus B)_\beta$.
\end{proof}

\subsection{Algebraic properties} \label{subsect: algebraic properties}

It will prove fruitful to study SGCs for their own sake, as they turn out to have interesting algebraic properties.

To motivate our investigation, let's begin by questioning the definition of a left SGC itself. In particular, condition (ii). Might it actually be the case that for every left SGC $\mathcal{C}$, $A + B \in \mathcal{C}$ implies that $A, B \in \mathcal{C}$?

\begin{example}
    Consider $\mathcal{C} = \{L \in LO\;|\;L \text{ has no maximal element}\}$. $\mathcal{C}$ is clearly closed under isomorphism and satisfies condition (ii). No sum of linear orders without a right endpoint has a right endpoint, therefore (iii) is also satisfied.

    That said, for example, $1 + \mathbb{Q} \in \mathcal{C}$, even though $1 \notin \mathcal{C}$.

    Note that this class appears to be dual to the class of well-orders in the following sense: given any linear order $A$, in order to find the decomposition $A \cong A_L + A_R$, where $A_L$ is the longest initial segment of $A$ in $\mathcal{C}$, we must find a final segment $A_R$ such that every one of its initial segments has a top point, in other words, such that every final segment of $(A_R)^*$ has a bottom point, which implies that $(A_R)^*$ is well-ordered. Therefore, we must find the decomposition $A^* \cong \alpha + \Bar{A}$ where $\alpha$ is the longest initial segment isomorphic to an ordinal. $\alpha^* \cong A_R$ and $\Bar{A}^* \cong A_L$ in this decomposition.
\end{example}

Let's isolate this notion of duality.

\begin{definition} \label{def: dual class}
    Given a left SGC $\mathcal{C}$, define its \textit{dual class} $\mathcal{C}^*$ to be the class of all orders $L$ such that $L^*$ has no initial segment in $\mathcal{C}$.
\end{definition}

\begin{remark} \label{rmk: dual to well-orders}
    If $\mathcal{W}$ is the class of well-orders, as shown above, we have $$\mathcal{W}^* = \{L \in LO\;|\;L \text{ has no maximal element}\}.$$
\end{remark}

\begin{proposition} \label{prop: dual of a left SGC is a left SGC}
    For any left SGC $\mathcal{C}$, $\mathcal{C^*}$ is also a left SGC.
\end{proposition}

\begin{proof}
    Since $\mathcal{C}$ is closed under isomorphisms, $\mathcal{C}^*$ is as well. 

    Say $A + B \in \mathcal{C}^*$, then $(A + B) ^* = B^* + A^*$ has no initial segment in $\mathcal{C}$, so specifically, $B^*$ has no initial segment in $\mathcal{C}$, therefore $B \in \mathcal{C}^*$. If $A, B \in \mathcal{C}^*$, then since $(A + B) ^* = B^* + A^*$, if $(A+B)^*$ had an initial segment in $\mathcal{C}$, it would have to be of the form $B^* + \Bar{A}^*$, for $\Bar{A}$ some non-empty final segment of $A$, but then by (ii) we must have $\Bar{A}^* \in \mathcal{C}$, a contradiction to $A^*$ having no initial segments in $\mathcal{C}$.

    Given any increasing sequence $A_0 \preceq A_1 \preceq \dots$ of elements of $\mathcal{C}^*$, let $A$ be its direct limit. If $A^*$ has some initial segment in $\mathcal{C}$, then $A^* \cong \Bar{A}^* + \hat{A}^*$, where $\Bar{A}^* \in \mathcal{C}$, but then for some $\alpha \in Ord$, $A_\alpha^* \cong \tilde{A}^* + \hat{A}^*$, where $\Bar{A}^* \cong R^* + \tilde{A}^*$ (i.e. $A \cong \hat{A} + \tilde{A} + R$), so $\tilde{A}^* \in \mathcal{C}$, a contradiction to $A_\alpha^* \in \mathcal{C}^*$.
\end{proof}

Note that we never used the fact that $\mathcal{C}$ is closed under limits in this proof, we only needed conditions (i), (ii) and 2.(i). Therefore, we only need a class satisfying conditions (i), (ii) and 2.(i) for its dual to be left sum-generating.

But this phenomenon isn't limited to the dual operation. Proving that $\mathcal{C}^\perp$ is sum-generating also only requires that $\mathcal{C}$ satisfy (i), (ii) and 2.(i).

\begin{example}
    Let $\mathcal{C} = \{ A \;|\; A \cong A' + \mathbb{Q} \text{ for some } A' \in LO\}$. 
    
    $\mathcal{C}$ satisfies (i), (ii) and 2.(i), but fails to satisfy 2.(ii), since for example, we may consider the sequence $A_\alpha = \alpha\times\mathbb{Q} = \mathbb{Q} + \dots + \mathbb{Q}$ for all $\alpha \in\omega_1$. This sequence is of the type in 2.(ii), and its right limit is $\omega_1 \times \mathbb{Q}$, but this order has no final copy of $\mathbb{Q}$, so it can't be in $\mathcal{C}$. 
    
    Despite this, $\mathcal{C}^*$ is left sum-generating by Proposition \ref{prop: dual of a left SGC is a left SGC}. 
    
    $A \in \mathcal{C}^*$ iff $A^*$ has no initial segment ending with an isomorphic copy of $\mathbb{Q}$ iff $A$ does not convexly embed $\mathbb{Q}$. Independently of our proposition, it is easy to see that this class is indeed left sum-generating (in fact, it is also right sum-generating). Note that this class differs from the one used to generate $+_s$ since, for example $\mathbb{R} \in \mathcal{C}^*$, even though the reals are not scattered.
\end{example}

\begin{example}
    Consider $\mathcal{S}$, the class of scattered orders. This is both a left and right SGC (in fact, it is the intersection of all classes that are both left and right SGCs and contain $1$\footnote{This follows by applying Hausdorff's characterization \cite{Hausdorff1908}.}). Notice how $^*$ and $^\perp$ behave on $\mathcal{S}$ (as a left SGC):
    
    \begin{itemize}
        \item $A \in \mathcal{S}^*$ iff every non-empty final segment of $A$ embeds $\mathbb{Q}$, while

        \item $A \in \mathcal{S}^\perp$ iff every non-empty initial segment of $A$ embeds $\mathbb{Q}$.
    \end{itemize}

    If we apply these operations, considering it a right SGC, this will flip. Further, note that while $\mathcal{S}$ is both a left and right SGC, $\mathcal{S}^*$ and $\mathcal{S}^\perp$ are only one of the two.
\end{example}

A unary inversion operator that is defined the same for all SGCs, regardless of whether they are a left or right SGC is the \textit{inverse}.

\begin{definition} \label{def: inverse of SGC}
    Given a (left or right) SGC $\mathcal{C}$, its \textit{inverse} $\mathcal{C}^{-1}$ is the (right or left) SGC all of whose elements are the inverse of elements of $\mathcal{C}$\footnote{If $(A, <_A)$ is a linear order its \textit{inverse} is the order $(A, <_A^*)$ such that $a<_A^*b$ iff $b<_Aa$.
    
    Unlike $^*$ and $^\perp$, the proof that $^{-1}$ outputs an SGCs does require the input to satisfy all the conditions of an SGC.}.
\end{definition} 

\begin{remark} \label{def: relation between involution operators}
    Note that for any SGC $\mathcal{C}$, $(\mathcal{C}^*)^\perp = \mathcal{C}^{-1}$.
\end{remark}

In the case of the scattered orders, we have $\mathcal{S}^{-1} = \mathcal{S}$.

Given $\mathcal{C}$ an SGC, is it the case that $\mathcal{C}^{**} = \mathcal{C}$? We certainly know this to be false if $\mathcal{C}$ satisfies (i), (ii) and 2.(i), but not 2.(ii), since in that case $\mathcal{C}^{**}$ will be sum-generating even though $\mathcal{C}$ is not.

\begin{proposition} \label{prop: inverse is involutive}
    If $\mathcal{C}$ is right sum-generating, $\mathcal{C}^{**} = \mathcal{C}$.
\end{proposition}

\begin{proof}
    $A \in \mathcal{C}^{**}$ iff no initial segment of $A^*$ is in $\mathcal{C}^*$ iff for every initial segment $B^*$ of $A^*$, $B^{**} = B$ has an initial segment in $\mathcal{C}$ iff every final segment $B$ of $A$ has an initial segment in $\mathcal{C}$.

    Let $A\in \mathcal{C}^{**}$. We construct an increasing sequence of linear orders in $\mathcal{C}$ whose direct limit is $A$. Treating $A$ as a final segment of itself, let $A_0$ be an initial segment of it that is in $\mathcal{C}$. By induction, if $A_\alpha$ is not $A$, let $B$ be a final segment of $A$ such that $A_\alpha + B \cong A$, let $A_{\alpha+1} = A_\alpha + \Bar{B}$, where $\Bar{B}$ is an initial segment of $B$ in $\mathcal{C}$. If $\lambda$ is a limit and we have constructed $A_\alpha$ for all $\alpha < \lambda$, let $A_\lambda$ be their direct limit. Since we do not allow count $\emptyset$ as a linear order, this construction has to eventually stop. $A = \varinjlim A_\alpha \in \mathcal{C}$.

    If $A \in \mathcal{C}$, given any decomposition $A = \Bar{A} + A'$, $A' \in \mathcal{C}$, so $A \in \mathcal{C}^{**}$.
\end{proof}

This proposition paired with the previous identity also gives us that $^{-1}$ and $^\perp$ commute on any SGC. In fact, all of $^*$, $^\perp$ and $^{-1}$ must commute, since it is easy to see that $^{-1}$ also has to satisfy $(\mathcal{C}^{-1})^{-1} = \mathcal{C}$.

The above proof also shows that if $\mathcal{C}$ satisfies (i), (ii) and 2.(i), then $\mathcal{C} \subseteq \mathcal{C}^{**}$. In fact, it will be the smallest left sum-generating class containing $\mathcal{C}$. This is easy to see by the first part of the above proof, where for any $A \in \mathcal{C}^{**}$ we constructed a sequence of elements of $\mathcal{C}$ whose right limit is $A$ and since any left sum-generating class has to be closed under right limits.

The intersection of two sum-generating classes will also be a sum-generating class, denote the intersection of $\mathcal{A}$ and $\mathcal{B}$ by $\mathcal{A}\cdot\mathcal{B}$. That said, the union doesn't need to be an SGC since it won't necessarily be closed under sums or right limits. What is the smallest sum-generating class that contains any two sum-generating classes?

Note that we may add two sum generation classes in the obvious way: let $\mathcal{A}\oplus\mathcal{B}$ be the smallest class containing $\mathcal{A}\cup\mathcal{B}$ that is closed under finite sums and isomorphism. This is clearly the smallest class satisfying (i), (ii) and 2.(i) and containing both $\mathcal{A}$ and $\mathcal{B}$. Let $\mathcal{A} + \mathcal{B} = (\mathcal{A} \oplus \mathcal{B})^{**}$, by our remarks following Proposition \ref{prop: inverse is involutive}, this is the smallest sum-generating class containing both $\mathcal{A}$ and $\mathcal{B}$. Any element of $\mathcal{A} + \mathcal{B}$ will be a well-ordered sum of orders of the form $A + B$ or $B + A$, where $A \in \mathcal{A}$ and $B \in \mathcal{B}$. The definition $\mathcal{A} + \mathcal{B} = (\mathcal{A} \oplus \mathcal{B})^{**}$ also works for right SGCs.

Further, $+$ and $\cdot$ are commutative, associative operations, and for any SGC $\mathcal{C}$, they witness the collection of all its sub-SGCs under $\subseteq$ being a bounded lattice. 

Using Boolean algebra notation, let $\mathbf{1} := LO$, $\mathbf{0} := \{\emptyset\}$.

Given any SGCs $\mathcal{A}, \mathcal{B}$, the following identities hold:

$$\mathcal{A}\cdot\mathbf{1} = \mathcal{A},\ \mathcal{A} \cdot \mathbf{0} = \mathbf{0}, \ \mathcal{A} + \mathbf{0} = \mathcal{A} \text{, and } \mathcal{A} + \mathbf{1} = \mathbf{1}.$$

\begin{remark}
    Generally, $^*$ and $^{-1}$, do not behave as one would hope with respect to these operations. For example $\mathcal{W}\cdot \mathcal{W}^*$ is the class of well-orders isomorphic to a non-zero limit ordinal, while $\mathcal{W}$ and $\mathcal{W}^{-1}$ have different chirality. But even when this doesn't constitute a problem, such as in the case of the scattered orders, $\mathcal{S} \cdot \mathcal{S}^{-1} = \mathcal{S} + \mathcal{S}^{-1} = \mathcal{S}$.
    
    On the other hand, $^\perp$ comes closer to behaving nicely. If $\mathcal{C}$ is any SGC,

    $$\mathcal{C} \cdot \mathcal{C}^\perp = \mathbf{0}, \ \mathcal{C} \oplus \mathcal{C}^\perp = \mathbf{1}.$$

    The latter fact comes from the observation that any linear order may be decomposed into it longest initial (or final) segment in $\mathcal{C}$ added to an element of $\mathcal{C}^\perp$.

    In general, by definition $\mathcal{C} \cap \mathcal{C}^\perp = \mathbf{0}$.
\end{remark}

Despite what one might expect, these operations do not distribute over each other. While $$\mathcal{A} \cdot \mathcal{B} + \mathcal{A} \cdot \mathcal{C} \subseteq \mathcal{A}\cdot(\mathcal{B} + \mathcal{C}) \text{ and } \mathcal{A} + (\mathcal{B} \cdot \mathcal{C}) \subseteq (\mathcal{A} + \mathcal{B}) \cdot (\mathcal{A} + \mathcal{C})$$ are immediate, we have counterexamples to the opposite inclusions.

If we let $\mathcal{A} = \langle1\rangle^*$ (the class of linear orders without a maximal element), $\mathcal{B} = \langle1\rangle$ (the well-orders), and $\mathcal{C} = \langle\mathbb{Q}\rangle$. Note that $\mathcal{A}\cdot(\mathcal{B} + \mathcal{C}) = \mathcal{A} \cdot \langle 1, \mathbb{Q}\rangle$ and $\mathcal{A} \cdot \mathcal{B} + \mathcal{A} \cdot \mathcal{C} = \langle \omega, \mathbb{Q}\rangle$, but $2 + \mathbb{Q} \in \mathcal{A} \cdot \langle 1, \mathbb{Q}\rangle$ even though it is not in $\langle \omega, \mathbb{Q}\rangle$.

In the other case, let $\mathcal{A} = \langle\mathbb{Q}\rangle,\ \mathcal{B} = \mathcal{S},\ \mathcal{C} = \langle \omega \rangle$\footnote{$\mathcal{C}$ is (up to order-type) the class of limit ordinals.}. $$\mathcal{A} + (\mathcal{B} \cdot \mathcal{C}) = \langle \mathbb{Q}, \omega\rangle \neq \mathcal{S} = (\mathcal{A} + \mathcal{B}) \cdot (\mathcal{A} + \mathcal{C}).$$

Therefore, we do not obtain a Boolean algebra on the SGCs with these operations. 

That said, the inversion operation do behave nicely with the $\textit{subclass of}$ relation. Say $\mathcal{A} \subseteq \mathcal{B}$, then $\mathcal{B}^* \subseteq \mathcal{A}^*$ and $\mathcal{B}^\perp \subseteq \mathcal{A}^\perp$, while $\mathcal{A}^{-1} \subseteq \mathcal{B}^{-1}$.

Despite the algebraic similarity between the lattice sub-SGCs of an SGC and the lattice of subspaces of a vector space, the former lattice is not modular. Consider $\mathcal{A} = \langle \omega \rangle$, $\mathcal{B} = \mathcal{S}$ and $\mathcal{C} = \langle1\rangle^*$. Even though, $\mathcal{A} \subseteq \mathcal{C}$ and $\mathbb{Z} \cong \omega^* + \omega \in (\mathcal{A} \oplus \mathcal{B}) \cdot \mathcal{C}$, despite $\omega^* + \omega \notin \mathcal{A} \oplus (\mathcal{B} \cdot \mathcal{C})$.

To conclude our discussion of the lattice structure of sub-SGCs, we fully characterize this lattice in the case of $\langle1\rangle$, illustrating an interesting algebraic fact about the class of well-orders as an SGC.

\begin{definition} \label{def: principal SGC}
    A sum-generating class $\mathcal{C}$ is \textit{principal} if it is generated by a single elements (i.e. there is an $A \in \mathcal{C}$ such that $\langle A \rangle = \mathcal{C}$).
\end{definition}

\begin{theorem} \label{th: the well-orders are a principal SGC}
    Every sum-generating subclass $\mathcal{C}$ of $\mathcal{W}$ is principal. In fact, $\mathcal{C}$ is generated by an additively indecomposable ordinal.
\end{theorem}

\begin{proof}
    Let $\alpha \in \mathcal{C}$ be smallest ordinal in the class. By Cantor Normal Form, the right additive indecomposability of AI ordinals (both proven in \ref{subsect: good sums of well-orders}), and property (ii) of sum-generating classes, $\alpha = \omega^\gamma$ for some $\gamma \in Ord$.

    Note that for any $\beta > \gamma$, $\omega^\beta \in \langle\alpha\rangle$, since $\omega^\beta = \omega^\gamma \cdot \omega^\delta$, which is the result of summing $\omega^\gamma$ indexed by $\omega^\delta$, where $\delta$ is such that $\gamma + \delta  = \beta$ (see lemma 2.25 of \cite{jech2014set}).

    If $\epsilon \in \mathcal{C}$, then every $\omega^\beta$ in its CNF is greater than $\alpha$ since the $\omega^\beta$ of smallest degree has to be in $\mathcal{C}$ by one application of (ii). Hence we must have all of $\epsilon$ in $\langle\alpha\rangle$ by closure under sums.
\end{proof}

\begin{joke}
    $Ord$ is a PID.
\end{joke}

The lattice of sub-SGCs of $\langle 1 \rangle$ will be linearly ordered and consist of classes $\langle \omega^\alpha\rangle$. We have, $\langle \omega^\alpha\rangle \subseteq \langle \omega^\beta \rangle$ iff $\beta \leq \alpha$. So, we have shown that this lattice is in fact isomorphic to $(Ord \cup \{\infty\}, \leq^*)$.

\begin{center}
\begin{tikzcd}[row sep=small]
	{\langle1\rangle} & 0 \\
	{\langle\omega\rangle} & 1 \\
	{\langle\omega^2\rangle} & 2 \\
	\vdots & \vdots \\
	{\langle\varepsilon_0\rangle} & {\varepsilon_0} \\
	\vdots & \vdots \\
	{\{0\}} & \infty
	\arrow["{\mathrel{\rotatebox{90}{$\subset$}}}", no head, from=1-1, to=2-1]
	\arrow["{\mathrel{\rotatebox{90}{$<^*$}}}", no head, from=1-2, to=2-2]
	\arrow["{\mathrel{\rotatebox{90}{$\subset$}}}", no head, from=2-1, to=3-1]
	\arrow["{\mathrel{\rotatebox{90}{$<^*$}}}", no head, from=2-2, to=3-2]
	\arrow["{\mathrel{\rotatebox{90}{$\subset$}}}", no head, from=3-1, to=4-1]
	\arrow["{\mathrel{\rotatebox{90}{$<^*$}}}", no head, from=3-2, to=4-2]
	\arrow["{\mathrel{\rotatebox{90}{$\subset$}}}", no head, from=4-1, to=5-1]
	\arrow["{\mathrel{\rotatebox{90}{$<^*$}}}", no head, from=4-2, to=5-2]
	\arrow["{\mathrel{\rotatebox{90}{$\subset$}}}", no head, from=5-1, to=6-1]
	\arrow["{\mathrel{\rotatebox{90}{$<^*$}}}", no head, from=5-2, to=6-2]
	\arrow["{\mathrel{\rotatebox{90}{$\subset$}}}", no head, from=6-1, to=7-1]
	\arrow["{\mathrel{\rotatebox{90}{$<^*$}}}", no head, from=6-2, to=7-2]
\end{tikzcd}
\end{center}

Further, $\langle \omega^\alpha\rangle + \langle \omega^\beta\rangle = \langle \omega^{\min(\alpha, \beta)}\rangle$, while $\langle \omega^\alpha\rangle \cdot \langle \omega^\beta\rangle = \langle \omega^{\max(\alpha, \beta)}\rangle$, so these operations will be distributive over the subclasses of the ordinals.

Another interesting example of a principal SGC is $\langle1, \mathbb{Q}\rangle = \langle\mathbb{Q}+1\rangle$. We can obtain $1$ in $\langle\mathbb{Q}+1\rangle$ by simply stripping it from $\mathbb{Q} + 1$, and $\mathbb{Q}$ by taking a countable sum to get $\omega\times(\mathbb{Q} + 1)$, which is a countable, dense linear order without endpoints, therefore isomorphic to $\mathbb{Q}$. This whole class won't have the same property as the well-orders, since it contains subclasses such as $\langle\omega, \mathbb{Q}\rangle$ which are not principal\footnote{If it were equal to some $\langle A \rangle$, then $A \cong B + \omega$ since it must have a final segment isomorphic to an ordinal in order to generate $\omega$ (you can't get well-orders by taking limits of something that is not well-ordered, the only way to get a well-order is to strip it out, or take limits of well-orders), but it can't be finite or else we could generate all the ordinals ($\langle1\rangle$), so it must be $\omega$ itself. That said, $\omega$ is not dense, therefore any limit of $B + \omega$ won't be either, so we can't generate $\mathbb{Q}$.}.

Given that we care about sums of linear orders, let's consider one more binary operation among SGCs that will allow us to introduce an important concept.

\begin{definition} \label{def: instance of a sum}
    Given two orders $A$ and $B$, $C$ is an \textit{instance of a sum} of $A$ and $B$ if there is some sum  $\oplus$ such that $A \oplus B = C$. Equivalently, $C$ is an instance of a sum of $A$ and $B$ if there are disjoint suborders $A', B' \subseteq C$ such that $A' \cup B' = C$ and $A \cong A'$, $B \cong B'$.
\end{definition}

\begin{definition} \label{def: shuffle sum}
    Given two SGCs $\mathcal{A}, \mathcal{B}$ of the same chirality, let their \textit{shuffle sum} $\mathcal{A} \boxplus \mathcal{B}$ be the SGC of the same chirality consisting of all instances of sums of elements of $\mathcal{A}$ with elements of $\mathcal{B}$.
\end{definition}

\begin{example}
    If $\alpha, \beta \in Ord$, $\langle \omega^\alpha \rangle \boxplus \langle \omega^\beta \rangle = \langle \omega^{\max(\alpha, \beta)} \rangle$. This follows by the fact that any instance of a sum of two well-orders is again a well-order, paired with the bounds on instances of sums of two ordinals given in Theorem \ref{th: min sum is minimal/bounds on instances of a sum of ordinals}.

    The shuffle sum may also yield quite unwieldy results, for example, $\langle \mathbb{Q}\rangle \boxplus \langle \mathbb{Q} \rangle$ contains $\mathbb{Q}(1, 2)$ -- the rationals, partitioned into two dense orders, one of which has every point replaced with $2$. We'll talk more about such orders and how they may be obtained as instances of sums of $\mathbb{Q}$ with itself in the next section.
\end{example}

We conclude this section by talking some more about the decomposition of orders generated by $^\perp$.

\begin{theorem}[Rigidity] \label{th: rigidity}
Given a left SGC $\mathcal{C}$, every linear order $A$ has a unique decomposition as $A \cong A_L + A_R$ with $A_L \in \mathcal{C}$ and $A_R \in \mathcal{C}^\perp$. 

This decomposition is preserved under isomorphism: if $f: A \xrightarrow[]{\sim} B$, then $B_L = f(A_L)$ and $B_R = f(A_R)$.
\end{theorem}

\begin{proof}
Existence is given by letting $A_L$ be the longest initial segment of $A$ in $\mathcal{C}$, and letting $A_R$ be the corresponding final segment.

If we prove uniqueness, we'll immediately get preservation under isomorphism by the proof of the forward direction of Theorem \ref{th: characterization of sum-generating classes with simple sums}.

Say $A_L + A_R \cong A \cong A_L' + A_R'$, where $A_L' \in \mathcal{C}$, $A_R' \in \mathcal{C}^*$. Let $g: A_L' + A_R' \xrightarrow{\sim} A_L + A_R$. $g(A_L') \cap A_R$ is a final segment of $g(A_L') \cong A_L'$, therefore is in $\mathcal{C}$, and it is also an initial segment of $A_R'$, hence must also be in $\mathcal{C}^\perp$, so it must be empty. Similarly, $g(A_R') \cap A_L$ must also be empty. Therefore, $g(A_L')  \subseteq A_L$ and $g(A_R')  \subseteq A_R$, since the map is onto, these must be equalities.
\end{proof}

\begin{corollary} \label{cor: simple sums are canonically regular}
    Simple sums are canonically regular.
\end{corollary}

We call the decomposition $A = A_L + A_R$ from the rigidity theorem the $\mathcal{C}$-decomposition of $A$. We also say that $\mathcal{C}$ determines the cut $A_L + A_R$ in $A$.

By uniqueness, given $A \in \mathcal{C}$ and $B \in \mathcal{C}^\perp$, the sum $A + B$ is its own $\mathcal{C}$-decomposition. In particular, if neither $A$ nor $B$ are empty, such a sum belongs to neither $\mathcal{C}$ nor $\mathcal{C}^\perp$. If $A'$ is a final segment of $A$ and $B'$ is an initial segment of $B$, then $A' \in \mathcal{C}$ and $B' \in \mathcal{C}^\perp$, so that $A' + B'$ is also a $\mathcal{C}$-decomposition. In this sense, cuts determined by $\mathcal{C}$ are invariant under ``zooming in'' around the cut.

\begin{proposition} \label{prop: reverse sum of two orders is in one of the SGCs}
If $A \in \mathcal{C}$ and $B \in \mathcal{C}^\perp$, then $B + A$ belongs to either $\mathcal{C}$ or $\mathcal{C}^\perp$.
\end{proposition}

\begin{proof}
Suppose not. Then the $\mathcal{C}$-decomposition of $\alpha + \beta \cong B + A$ for some non-empty $\alpha \in \mathcal{C}$ and $\beta \in \mathcal{C}^\perp$, let $f: \alpha + \beta \xrightarrow[]{\sim} B + A$ witness this decomposition. We cannot have $f(\alpha) \in \mathcal{C}$ be an initial segment of $B \in \mathcal{C}^\perp$, so $f(\alpha)$ must contain $B$ as an initial segment. This implies that $f(\beta) \in \mathcal{C}^\perp$ is a final segment of $A \in \mathcal{C}$, which is a contradiction since $f(\beta) \neq \emptyset$.
\end{proof}

\begin{corollary} \label{cor: SGC decomposition of sum of two orders}
Suppose $A$ and $B$ are orders with $\mathcal{C}$-decompositions $A = A_L + A_R$ and $B = B_L + B_R$ respectively. Then the $\mathcal{C}$-decomposition of $A + B$ is either $(A_L+ A_R + B_L) + B_R$ or $A_L + (A_R + B_L + B_R)$.
\end{corollary}

\begin{proof}
By Proposition \ref{prop: reverse sum of two orders is in one of the SGCs}, we have that either $A_R + B_L \in \mathcal{C}$ or $A_R + B_L \in \mathcal{C}^\perp$. In the first case, the $\mathcal{C}$-decomposition of $A + B$ is $(A_L + A_R + B_L) + B_R$, and in the second, $A_L + (A_R + B_L + B_R)$.
\end{proof}

There is also an analogue of Proposition \ref{prop: reverse sum of two orders is in one of the SGCs} for direct limits of initial and final segments.

\begin{proposition}\label{prop: contrary right and left limits are in one of the classes} Let $\mathcal{C}$ be a right SGC, then
\begin{enumerate}
    \item[(i)] if $A$ is a right limit of orders in $\mathcal{C}$, then $A$ belongs to either $\mathcal{C}$ or $\mathcal{C}^\perp$; dually,
    \item[(ii)] if $B$ is a left limit of orders in $\mathcal{C}^\perp$, then $B$ belongs to either $\mathcal{C}$ or $\mathcal{C}^\perp$.
\end{enumerate}
\end{proposition}

\begin{proof}
    We consider the first case, the proof of the second is symmetric. 
    
    Say $A = \bigcup_{\alpha < \lambda} A_\alpha$, where $A_\alpha \in \mathcal{C}$ and if $\alpha < \beta$, then $A_\alpha$ is an initial segment of $A_\beta$. Without loss of generality, we may assume that $\lambda$ is a regular cardinal, it it is not, consider a sequence with the same right limit of length $\mbox{cf}(\lambda)$.
    
    Let $A \cong A_L + A_R$ be its $\mathcal{C}^\perp$-decomposition. If $A_L, A_R \neq \emptyset$, then for any large enough $\alpha$, $A_L$ will be a strict initial segment of $A_\alpha$, therefore $A_L \in \mathcal{C}$, which is a contradiction to $A_L \neq \emptyset$.
\end{proof}

\subsection{Commutativity of simple sums} \label{subsect: commutativity of simple sums}

\begin{definition} \label{def: commutativity}
    A sum $\oplus$ is \textit{commutative} if for any orders $A, B$, $$A \oplus B \cong B \oplus A.$$
\end{definition}

The next section will be entirely dedicated to commutativity, which is the last new algebraic property we will consider. To underscore the novelty of this concept in our investigation, we prove that no simple sum is commutative.

\begin{theorem} \label{th: simple sum is not commutative}
    If $\mathcal{C}$ is a left SGC, $\oplus_\mathcal{C}$ is not commutative.
\end{theorem}

We have developed enough tools so as to have multiple ways to approach this theorem. In fact, before reading our proof, we encourage the reader to attempt carrying out this proof however they see fit.

A proof using the distinguishing lemma is quite efficient, but we opt for one that holds in any subclass of $LO$ closed under isomorphism and containing $1, \omega, \omega^*, \omega+1$, and $(\omega+1)^*$.

\begin{proof}
    Assume $\oplus_\mathcal{C}$ is commutative.

     By the rigidity theorem, either $1 \in \mathcal{C}$ or $1 \in \mathcal{C}^\perp$. Assume the former first. By closure under ordinal sums $\omega \in \mathcal{C}$, so

    $$ \omega + 1 = 1 \oplus_\mathcal{C} \omega \cong \omega \oplus_\mathcal{C} 1 = 1 + \omega \cong \omega,$$

    which is a contradiction. If $1 \in \mathcal{C}^\perp$, then, similarly, $\omega^* \in \mathcal{C}$, but

    $$\omega^* \oplus_\mathcal{C} 1 = \omega^* + 1 \cong \omega^* \not\cong (\omega+1)^* \cong 1 + \omega^* \cong 1 \oplus_\mathcal{C} \omega^*.$$
\end{proof}

\section{Good sums of linear orders} \label{sect: good sums of linear orders}

Recall that our original theorem stating that the class of linear orders has no coproducts could be proven using Lemma \ref{lm: coproduct is isomorphic to the usual sum} by noting that, in general, the usual sum is not commutative. From an algebraic perspective, the lack of existence of coproducts in $LO$ can be seen as stemming from a failure of commutativity in candidate coproducts (sums). Given that one of our initial motivations was to generalize coproducts in the setting of linear orders, it would be natural to consider commutative sums, the sums which, from an algebraic perspective, most resemble coproducts. Pairing commutativity with our previous reasonable properties yields the following definition.

\begin{definition} \label{def: good sums}
    A sum on a class $C$ of linear orders is said to be \textit{good} on $C$  if it is regular, associative, and commutative on $C$.

    We say that a class $C$ is \textit{good} if there exists a good sum on $C$. 
\end{definition}

This section will be dedicated to studying good sums on different classes of orders. We will begin by investigating and characterizing good sums on the ordinals and immediate extensions thereof as an illustration of how good sums may be defined in highly structured contexts. We will then begin discussing how good sums may be defined in non-structured settings, leading to the idea of \textit{complicated classes}, classes of orders within which there is maximal freedom to define sums. Ultimately, in \ref{sect: revisiting the properties of sums} we will demonstrate how this investigation of good sums on subclasses of $LO$ can be paired with our work on SGCs to show that the original properties of sums we considered in \ref{subsect: some reasonable properties of sums} are independent from regularity and associativity.

\subsection{Good sums on the well-orders} \label{subsect: good sums of well-orders}

We begin by showing that the class of well-orders is good. Working up to order-type, we will refer to a well-order by the unique ordinal isomorphic to it.

An important property of the class of well-orders (also shared with the class of scattered orders) is that it is closed under arbitrary instances of sums. This gives us a lot of freedom when defining new sums, since we don't have to worry about the possibility of leaving $\mathcal{W}$ when defining a sum.

In order to define our first good sum on the ordinals, we will need to develop the basic structure theory of the class of ordinals. The most important piece of structure theory is Cantor's normal form theorem.

\begin{theorem}[Cantor Normal Form \cite{cantor1895bei}] \label{th: cantor normal form} Any non-zero ordinal $\alpha$ may be  uniquely represented as

$$\alpha = \omega^{\beta_1} \cdot k_1 + \dots + \omega^{\beta_n} \cdot k_n$$

where $n > 1$, $\alpha \geq \beta_1 > \dots > \beta_n$, and for all $i \in \{1, \dots, n\}$, $0 < k_i < \omega$.
    
\end{theorem}

\begin{remark}
    Cantor normal form is interesting from the order-theoretic point of view, since it decomposes any well-order uniquely into additively indecomposable well-orders. As we will see in this section, the normal form makes defining good sums on $Ord$ much easier.
\end{remark}

\begin{definition} \label{def: additive indecomposability}
    A linear order $A$ is said to be \textit{additively indecomposable} (AI) if whenever $A \cong B + C$, either $A \preceq B$ or $A \preceq C$ (equivalently $A \equiv B$ or $A \equiv C$).
\end{definition}

These summands also satisfy a stronger notion of indecomposability.

\begin{definition} \label{def: indecomposability}
    A linear order $A$ is \textit{strongly indecomposable} if whenever $A$ is an instance of a sum of $B$ and $C$, either $A \equiv B$ or $A \equiv C$.
\end{definition}

The following two propositions are well-known facts about the ordinals, for proofs see \cite{rosenstein1982linear}.

Firstly, recall that $\preceq$ coincides with $\leq$, and $\equiv$ with $=$ on the class of ordinals.

\begin{proposition} \label{prop: for ordinals embeds in is equivalent to less} The following two statements are true for all $\alpha, \beta \in Ord$:

    \begin{itemize}
        \item[(i)] $\alpha \preceq \beta$ iff $\alpha \leq \beta$;
        \item[(ii)] $\alpha \equiv \beta$ iff $\alpha = \beta$.
    \end{itemize}
\end{proposition}

We can also give a concrete characterizations of the AI ordinals.

\begin{theorem} \label{th: ordinal is additively indecomposable iff it is power of omega}
    A non-zero ordinal $\alpha$ is AI iff it is of the form $\omega^\beta$ for some $\beta \in Ord$.
\end{theorem}

It turns out that both notions of indecomposability coincide for the well-orders, so this proposition also characterizes the strongly indecomposable orders.

\begin{proposition} \label{prop: ordinal is additively indecomposable iff it is strongly indecomposable}
    A non-zero ordinal $\alpha$ is AI iff it is strongly indecomposable.
\end{proposition}

At the end of this section, we will show how a proof of this well known fact can be derived from our bounds on instances of sums of two ordinals.

The decomposition into strongly indecomposable orders furnished by CNF accounts for a lot of unexpectedly nice properties of the theory of sums on the ordinals.

For example, attempting to bound the number of non-isomorphic instances of a sum of two infinite well-orders, a na\"ive approach would yield that there are infinitely many possible orders. Due to CNF, the number is actually finite.

\begin{theorem} \label{th: finitely many instances of sum of two ordinals}
    For any ordinals $\alpha, \beta$, there are (up to isomorphism) only finitely many instances of a sum of $\alpha$ and $\beta$.
\end{theorem}

This is not hard to show, but a nice bijective proof is given in theorem 1.38 of \cite{Toulmin1954ShufflingOA}.

We will prove that $\mathcal{W}$ is a good class by making use to the \textit{Hessenberg sum}. This is a commutative sum on the ordinals first studied by Gerhard Hessenberg in 1906 \cite{hessenberg}, it is also called the \textit{natural sum}, a name warranted by pervasiveness in various settings.

\begin{definition} \label{def: hessenberg sum}
    Given any two ordinals $\alpha, \beta$, let their expanded out Cantor normal forms be

    \[\alpha = \omega^{\beta_1} + \dots + \omega^{\beta_n}\]

    and

    \[\beta = \omega^{\gamma_1} + \dots + \omega^{\gamma_m}.\]

    The \textit{Hessenberg sum} of $\alpha$ and $\beta$ is the ordinal
    
    $$\alpha \# \beta = \omega^{\delta_1} + \dots + \omega^{\delta_{n + m}},$$

    where $\delta_1\leq\dots\leq\delta_{n+m}$ is a non-decreasing enumeration of $\{\beta_1, \dots, \beta_n, \gamma_1, \dots, \gamma_m\}$.
\end{definition}

The Hessenberg sum is clearly a sum in our sense. Further, it is easily seen to be commutative and associative, thereby good. While this definition is quite constructive, there is a more natural way to define this sum (we will see more in the following sections).

\begin{proposition} \label{prop: hessenberg sum is greater than any instance of a sum}
    If $\alpha \in Ord$ is an instance of a sum of $\beta, \gamma \in Ord$, then $\alpha \leq \beta \# \gamma$.
\end{proposition}

The intuition behind this folkloric result is that the Hessenberg sum provides us with the only way to combine two ordinals without any absorption.

To prove it, we consider Carruth's axioms for natural sums as given in \cite{carruth1942arithmetic}\footnote{Note that a sum in Carruth's sense is just a binary operation on the ordinals; it need not be a sum in our sense.}.

\begin{definition} \label{def: natural ordinal operation}
    (Carruth) A binary operation $\oplus: Ord \times Ord \rightarrow Ord$ is a \textit{natural ordinal operation} iff for any $\alpha, \beta, \gamma \in Ord$

    \begin{enumerate}
        \item $\alpha \oplus \beta = \beta \oplus \alpha$;
        \item $(\alpha \oplus \beta) \oplus \gamma = \alpha \oplus (\beta \oplus \gamma)$;
        \item $\alpha \oplus 0 = \alpha$;
        \item $\gamma \oplus \alpha > \gamma \oplus \beta$ iff $\alpha > \beta$.
    \end{enumerate}
\end{definition}

\begin{remark}
    As Carruth points out, the Hessenberg sum is the unique natural ordinal operation that satisfies the condition that for any ordinals $\alpha \geq \beta$ and $m, n < \omega$
    $$\omega^\alpha \cdot m + \omega^\beta \cdot n = \omega^\alpha \cdot m \oplus \omega^\beta \cdot n$$
\end{remark}

\begin{example}
    For a natural ordinal operation which isn't a sum, consider the operation $\oplus$ defined by

    \[
        \alpha \oplus \beta = \begin{cases}
            \alpha, \text{ if } \beta = 0\\
            \beta, \text{ if } \alpha = 0\\
            \alpha \# \beta\# 1, \text { if } \alpha, \beta \neq 0
        \end{cases}
    \]

    While this satisfies all of Carruth's axioms, it isn't a sum since, for example, $1 \oplus 1 = 3$.
\end{example}

\begin{theorem} \label{th: natural ordinal operations are greater than any instance of a sum}
    Suppose $\oplus$ is a natural ordinal operation. If $\gamma$ is an instance of a sum of two ordinals $\alpha, \beta$, then $\gamma \leq \alpha \oplus \beta$.
\end{theorem}

\begin{proof}
    By induction on $(\alpha, \beta)$. The statement holds for $\alpha = 0$ and any $\beta$. The statement also holds for $\beta = 0$ and any $\alpha$. Suppose that it holds for all pairs $(\delta, \varepsilon)$ such that $\delta \leq \alpha$ and $\varepsilon \leq \beta$ and at least one of these inequalities is strict. We show that it holds for $(\alpha, \beta)$.

    Let $\gamma = \gamma_\alpha \cup \gamma_\beta$ be the witnesses to $\gamma$ being an instance of a sum of $\alpha$ and $\beta$. Suppose $\gamma > \alpha \oplus \beta$, and let $\mbox{tp}(\gamma_\alpha \cap \alpha \oplus \beta) = \gamma'_\alpha$ and $\mbox{tp}(\gamma_\beta \cap \alpha \oplus \beta) = \gamma'_\beta$ (treating $\alpha \oplus \beta$ as an initial segment of $\gamma$), so that $\alpha \oplus \beta$ is an instance of a sum of $\gamma'_\alpha$ and $\gamma'_\beta$. If either $\gamma'_\alpha < \alpha$ or $\gamma'_\beta < \beta$ then by induction we have $\alpha \oplus \beta \leq \gamma'_\alpha \oplus \gamma'_\beta$, contradicting (4) in \Cref{def: natural ordinal operation}. Thus $\gamma'_\alpha = \alpha$ and $\gamma'_\beta = \beta$. But since $\alpha \oplus \beta$ is strictly initial in $\gamma$, either $\gamma'_\alpha$ embeds in strict initial segment of $\gamma_\alpha$ (which is isomorphic to $\alpha$) or $\gamma'_\beta$ embeds in a strict initial segment of $\gamma_\beta$ (isomorphic to $\beta$), a contradiction in either case. 
\end{proof}

In the case that $\oplus$ is the Hessenberg sum, this theorem gives us a proof of Proposition \ref{prop: hessenberg sum is greater than any instance of a sum}. On the other hand, since the Hessenberg sum is a sum in our sense, we also get:

\begin{corollary} \label{cor: natural ordinal operation is greater than hessenberg sum}
    Let $\alpha, \beta$ be any ordinals, then for any natural ordinal operation $\oplus$, $$\alpha \# \beta \leq \alpha \oplus \beta$$
\end{corollary}

Theorem \ref{th: natural ordinal operations are greater than any instance of a sum} thus also tells us that the Hessenberg sum is the unique sum which is a natural ordinal operation. One may wonder whether it is possible to show that the Hessenberg sum is the only good sum on $Ord$ by making use of this result. Surely, if condition (4) of Definition \ref{def: natural ordinal operation} holds for all good sums this characterization would follow.

This is not the case, with there in fact being many different good sums on $Ord$. In fact, one of the sums will yield the lower bound missing from Proposition \ref{prop: hessenberg sum is greater than any instance of a sum}.

To motivate these new classes of sums we begin by asking: how many possible instances of a sum are there for $\gamma = \omega^\alpha$ and $\delta = \omega^\beta$?

If $\alpha = \beta$, there are two. The first is the usual sum, where we obtain an order isomorphic to $\omega^\alpha \cdot 2$. The second one being where we split up the orders into infinitely many pieces to get an order $A = \omega^\alpha\; \sqcup\; \omega^\alpha$ with $<_A$ defined by

\[  
    (\varepsilon, i) < (\zeta, j) \text{ iff }
    \begin{cases}
        \varepsilon < \zeta \text{, or }\\
        \varepsilon = \zeta \text{ and } {i < j}
    \end{cases}
\]

This order is isomorphic to $2 \cdot \omega^\alpha = \omega^\alpha$. Visually, we get any element of $\delta$ and put it right after the corresponding element of $\gamma$, interweaving $\gamma$ and $\delta$ maximally.

On the other hand, if $\beta < \alpha$ the only possibilities will be $\omega^\alpha + \omega^\beta$ and $\omega^\alpha$.

In general, if we have $\omega^\alpha \cdot k_1$ and $\omega^\alpha \cdot k_2$ and if $k_1 \leq k_2$, we may shuffle in $\omega^\alpha \cdot k_1$ point-wise into the initial segment of $\omega^\alpha \cdot k_2$ isomorphic to $\omega^\alpha \cdot k_1$, to get a instance of a sum equal to $\omega^\alpha \cdot k_2$. This is the most efficient way of adding these two elements. Otherwise, we can choose how much to maximally shuffle in, leading to the conclusion that the only possible sums are $\omega^\alpha \cdot (k_2 + i)$, for $i < k_1$.

Using the absorption of ordinals of the form $\omega^\alpha$ into themselves and CNF, we can define a new good sum as follows.

\begin{definition} \label{def: lcm sum}
    If $\alpha = \sum_{i = 1} ^ n\omega^{\gamma_i} \cdot k_i + k$ and $\beta = \sum_{i = 1} ^ n\omega^{\gamma_i} \cdot m_i + m$ where $\gamma_1 > \gamma_2 > \dots > \gamma_n > 0$, define the \textit{lcm sum} of $\alpha$ and $\beta$ by
    $$\alpha+_{lcm}\beta = \sum_{i = 1} ^ n\omega^{\gamma_i} \cdot \max\{k_i, m_i\} + k + m$$
\end{definition}

\begin{remark}
    This sum is not a natural ordinal operation since it fails to satisfy condition (4) of \Cref{def: natural ordinal operation}. There are two ways to see this:

    \begin{itemize}
        \item[-] directly, by noting that $\omega\cdot3 +_{lcm} \omega \cdot2 = \omega \cdot 3 = \omega \cdot 3 +_{lcm} \omega$, even though $\omega \cdot 2 > \omega$;

        \item[-] or by using Corollary \ref{cor: natural ordinal operation is greater than hessenberg sum} since, for example, $\omega\cdot3 +_{lcm} \omega \cdot2 = \omega \cdot 3 < \omega \cdot 5 = \omega\cdot3 \# \omega \cdot2$.
    \end{itemize}

    That said, it is straightforward to prove that it is a good sum.
\end{remark}

We may generate an infinite amount of non-isomorphic good sums on $Ord$ by varying the definition of $+_{lcm}$. For example, for any ordinal $\alpha$, we may dictate that the sum $\oplus_\alpha$ act like the Hessenberg sum everywhere (i.e. summing coefficients next to all $\omega^\gamma$s) except at some $\omega^\alpha$ where it should take the maximum of the coefficients. Implicitly, these sums are variations on the sifted sum from Example \ref{ex: global extension of the Hessenberg sum}, with each $\oplus_\alpha$ being either the usual sum, or a sum interleaving the two orders.

These sums do not absorb as efficiently as we would like them to in order to get our lower bound. In particular, given for example $\omega^3 \cdot 2 + \omega$ and $\omega^3 \cdot 3 + \omega^2$ we would like to have our sum dynamically join the $\omega^3$s in the most efficient way possible, i.e. how the lcm sum does, while also absorbing $\omega$ into $\omega^2$ to yield $\omega^3 \cdot 3 + \omega^2$.

This is achieved by the following sum.

\begin{definition} \label{def: dynamic sum}
    The \textit{dynamic sum} of ordinals $\alpha = \sum_{i = 1} ^ n\omega^{\gamma_i} \cdot k_i$ and $\beta = \sum_{i = 1} ^ m\omega^{\delta_i} \cdot l_i$ (in CNF) inductively:

    \begin{itemize}
        \item[(i)] $\alpha * 0 = 0 * \alpha = \alpha$;
        \item[(ii)] If $\gamma_1 > \delta_1$, then $\alpha * \beta = \beta * \alpha = \alpha$;
        \item[(iii)] If $\gamma_1 = \delta_1 \neq 0$, then $\alpha * \beta = \beta * \alpha = \omega^{\gamma_1} \cdot \max\{k_1, l_1\} + (\sum_{2 < i < n} ^{} \omega^{\gamma_i} \cdot k_i * \sum_{2 < i < n} ^{}\omega^{\delta_i} \cdot l_i)$. If $\gamma_1 = \delta_1 = 0$, let $\alpha * \beta = k_1 + l_1$.
    \end{itemize}
\end{definition}

\begin{proposition} \label{prop: dynamic sum is a good sum}
    $*$ is a good sum on $Ord$.
\end{proposition}

\begin{proof}
    That this operation is a sum follows by the previous description of infinite ordinal shuffles.
    
    Commutativity can be shown inductively and follows by definition.

    For associativity, given $\alpha, \beta, \delta \in Ord$, if any one of them has a greater degree than the rest, then the result will be that ordinal, no matter the arrangement of association. Next, look at the longest continuous chain of $\omega^\gamma$s for which all the ordinals have non-zero coefficients and note that the result of the whole sum will take the maximum of the coefficients in the three ordinals for each of these elements, this is done associatively. Wherever the chain stops, the next elements in the CNF must have disagreements of degree, so we simply consider the $*$ sum of whichever two segments have equal greatest degree, or take whichever segment has greatest degree.
\end{proof}

But this sum still won't be minimal, for example, when adding $\alpha = \omega^2 \cdot 3 + \omega \cdot 2$ and $\beta = \omega^2 \cdot 2 + \omega \cdot 4$, we are going to get $\alpha * \beta = \omega^2 \cdot 3 + \omega \cdot 4$, even though we could have spared ourselves the $\omega \cdot 4$ at the end by absorbing all of $\beta$ into $\omega^2 \cdot 3$ (this equals $\omega ^2 \cdot 2 + \omega^2$, so we may combine the $\omega^2\cdot2$s together and place $\omega \cdot 4$ before $\omega^2$ so that it gets absorbed). Incorporating this modification, yields the true lower bounding sum.

\begin{definition} \label{def: min sum}
    If $\alpha, \beta \in Ord$, say $\alpha = \omega\cdot\alpha' + n_\alpha$ and $\beta = \omega \cdot \beta' + n_\beta$, define the \textit{min sum} as

    \[
    \alpha +_{min} \beta = 
    \begin{cases}
        \max\{\alpha, \beta\}, \text{ if } \alpha' \neq \beta',\\
        \alpha + n_\beta, \text{ else }
    \end{cases}
    \]
\end{definition}

This is also another example of a good sum.

\begin{proposition} \label{prop: min sum is a good sum}
    $+_{min}$ is a good sum on $Ord$.
\end{proposition}

\begin{proof}
    If $\alpha' \neq \beta'$, w.l.o.g. $\alpha' < \beta'$, and $\alpha + \beta = \max\{\alpha, \beta\}$. Otherwise, we may shuffle $\omega \cdot \beta'$ into $\omega \cdot \alpha'$ and add the finite coefficients. This proves that $+_{min}$ is a sum.

    It is commutative by definition. 
    
    Say we want to add $\alpha \geq \beta \geq \delta$. We have three ways to do so:
    $$(\alpha +_{min} \beta) +_{min}\delta,\; (\alpha +_{min} \delta) +_{min}\beta,\; (\beta +_{min} \delta) +_{min}\alpha$$
    There are three important cases to consider for associativity:

    \begin{enumerate}
        \item $\alpha' > \beta'$, in which case the sum will be $\alpha$;

        \item $\alpha' = \beta' = \delta'$, then all three instances of association will equal $\alpha + n_\beta + n_\delta$;

        \item $\alpha' = \beta' > \delta'$, where all the sums will be $\alpha + n_\beta$.
    \end{enumerate} \end{proof}

\begin{remark}
    Observe that this sum behaves the same as the cardinal sum on the cardinal numbers.
\end{remark}

We finally get the missing bound in Proposition \ref{prop: hessenberg sum is greater than any instance of a sum}.

\begin{theorem} \label{th: min sum is minimal/bounds on instances of a sum of ordinals}
    Let $\gamma \in Ord$ be an instance of a sum of $\alpha, \beta \in Ord$, then
    $$\alpha+_{min}\beta \leq \gamma \leq \alpha \# \beta$$
\end{theorem}

\begin{proof}
    The upper bound was Proposition \ref{prop: hessenberg sum is greater than any instance of a sum}.

    First, let $\alpha = \omega \cdot \alpha ' + n_\alpha$, $\beta = \omega \cdot \beta' + n_\beta$ and $\delta = \omega\cdot \delta' + n_\delta$ be arbitrary ordinals. We show that
    
    $$(\alpha +_{min} \beta) + \delta \geq \alpha +_{min} (\beta +\delta).$$

    There are three cases to consider:

    \begin{enumerate}
        \item If $\alpha' > \beta'$, then $(\alpha +_{min} \beta) + \delta = \alpha + \delta$; and either $\delta' > \beta'$, so $\beta +\delta = \delta$ and we have equality, or $\delta' \leq \beta'$, in which case either $\beta' + \delta' > \alpha'$ so $\alpha +_{min} (\beta +\delta) = \beta + \delta \leq \alpha + \delta$, or $\beta' + \delta' < \alpha'$ so $\alpha +_{min} (\beta +\delta) = \alpha$, or $\beta' + \delta' = \alpha'$ so $\alpha +_{min} (\beta +\delta) = \alpha + n_\beta + n_\delta < \alpha + \delta$, since if $\delta < \omega$, then $\delta' = 0$ and $\beta' + \delta' = \beta' = \alpha'$, a contradiction.

        \item If $\alpha' < \beta'$, then $(\alpha +_{min} \beta) + \delta = \beta + \delta = \alpha +_{min} (\beta +\delta)$.

        \item If $\alpha' = \beta'$, then either $\delta \geq \omega$ and $\beta' + \delta' > \alpha'$, so 
        $$(\alpha +_{min} \beta) + \delta =\alpha + n_\beta + \delta = \alpha + \delta = \beta + \delta =  \alpha +_{min} (\beta +\delta)$$
        or $\delta < \omega$ and
        $$(\alpha +_{min} \beta) + \delta =\alpha + n_\beta + n_\delta =  \alpha +_{min} (\beta +\delta)$$
    \end{enumerate}

    Using this inequality, we get that for any $\alpha, \beta, \gamma, \delta \in Ord$

    $$(\alpha \,+_{min}\, \beta) + \gamma + \delta = ((\alpha \,+_{min}\, \beta) + \gamma) + \delta
\geq ((\alpha + \gamma) \,+_{min}\, \beta) + \delta 
\geq (\beta + \delta) \,+_{min}\, (\alpha + \gamma)$$
    
    Back to our original setup: since $\gamma$ is an instance of a sum of $\alpha$ and $\beta$, the following decomposition
    
    $$\gamma = \sum_{\varepsilon<\delta} (\alpha_\varepsilon + \beta_\varepsilon), \; \alpha = \sum_{\varepsilon < \delta} \alpha_\varepsilon,\; \beta = \sum_{\varepsilon < \delta} \beta_\varepsilon$$
    
    exists, where without loss of generality, we may assume that $\alpha_0 \neq 0$.

    By induction on $\nu \leq \delta$, we show that 

    $$\sum_{\varepsilon<\nu} (\alpha_\varepsilon + \beta_\varepsilon) \geq \sum_{\varepsilon < \nu} \alpha_\varepsilon +_{min} \sum_{\varepsilon < \nu} \beta_\varepsilon$$

    The base case is immediate. If it holds for $\phi$, then for $\nu = \phi+1$

    $$\sum_{\varepsilon<\nu} (\alpha_\varepsilon + \beta_\varepsilon) = \sum_{\varepsilon<\phi} (\alpha_\varepsilon + \beta_\varepsilon) + (\alpha_\nu + \beta_\nu) \geq (\sum_{\varepsilon < \phi} \alpha_\varepsilon +_{min} \sum_{\varepsilon < \phi} \beta_\varepsilon) + (\alpha_\nu + \beta_\nu) \geq \sum_{\varepsilon < \nu} \alpha_\varepsilon +_{min} \sum_{\varepsilon < \nu} \beta_\varepsilon$$

    (using the inequality for the last step).

    If $\nu \in Lim$, then

    $$\sum_{\varepsilon<\nu} (\alpha_\varepsilon + \beta_\varepsilon) = \lim_{\phi < \nu} \sum_{\varepsilon<\phi} (\alpha_\varepsilon + \beta_\varepsilon) \geq \lim_{\phi < \nu} (\sum_{\varepsilon < \phi} \alpha_\varepsilon +_{min} \sum_{\varepsilon < \phi} \beta_\varepsilon)$$

    so this time we have to show that $$\lim_{\phi < \nu} (\sum_{\varepsilon < \phi} \alpha_\varepsilon +_{min} \sum_{\varepsilon < \phi} \beta_\varepsilon) \geq \sum_{\varepsilon < \nu} \alpha_\varepsilon +_{min} \sum_{\varepsilon < \nu} \beta_\varepsilon$$

    Assume this is false, then we must have 
    \begin{align*}
        \sum_{\varepsilon < \nu} \alpha_\varepsilon +_{min} \sum_{\varepsilon < \nu} \beta_\varepsilon > \lim_{\phi < \nu} (\sum_{\varepsilon < \phi} \alpha_\varepsilon +_{min} \sum_{\varepsilon < \phi} \beta_\varepsilon) \geq \begin{cases}
            \lim_{\phi < \nu} \sum_{\varepsilon < \phi} \alpha_\varepsilon = \sum_{\varepsilon < \nu} \alpha_\varepsilon,\\ \lim_{\phi < \nu}\sum_{\varepsilon < \phi} \beta_\varepsilon = \sum_{\varepsilon < \nu} \beta_\varepsilon
        \end{cases} 
    \end{align*}

    But then, the left hand side can't be the maximum of the summands, so by definition both summands must be successor ordinals. But then we get a contradiction to the inductive hypothesis since there is some ordinal $\mu < \nu$ such that for all $\varepsilon > \mu$ both $\alpha_\varepsilon$ and $\beta_\varepsilon$ are $0$.
\end{proof}

This result is quite remarkable: we are bounding things that may or may not be the result of a reasonable (e.g. associative) sum with bounds which are not only sums, but associative and commutative.

We conclude by using these bounds to prove the characterization of the strongly indecomposable ordinals.

\begin{proof}[Proof of Proposition \ref{prop: ordinal is additively indecomposable iff it is strongly indecomposable}]
    We need to show that for any $\gamma$, $\omega^\gamma$ is strongly indecomposable. The case of $\gamma = 0$ is immediate, so assume $\gamma > 0$.
    
    Say $\omega^\gamma$ is an instance of a sum of $\alpha$ and $\beta$. Without loss of generality, if $\omega \leq \omega^\delta\cdot m + \dots =\alpha \leq \beta = \omega^\delta\cdot n + \dots$, then 
    $$\omega^\delta\cdot n + \dots = \alpha +_{min} \beta \leq \omega^\gamma \leq \alpha \# \beta = \omega^\delta \cdot (m + n) + \dots,$$
    so $\omega^\delta \leq \omega^\gamma$, assume this is not equality, then 
    $$\alpha \# \beta = \omega^\delta \cdot (n + m) + \dots < \omega^\delta\cdot (n + m + 1) < \omega^{\delta + 1} \leq \omega^\gamma \leq \alpha \# \beta,$$
    a contradiction. So $\delta = \gamma$, and $\omega^\gamma \preceq\beta$ by Proposition \ref{prop: for ordinals embeds in is equivalent to less}. In fact, we must have $\omega^\gamma = \beta$, or else $\omega^\gamma < \beta \leq \alpha +_{min} \beta$.
\end{proof}

To conclude this section, we prove that an intuitive characterization of the Hessenberg sum based on our previous discussion of how absorption can occur, fails. We do so by constructing a new sum. 

\begin{definition} \label{def: splitting summand into kappa pieces}
    Given some cardinal $\kappa$, we say a sum $\oplus$ on a class $C$ \textit{splits summands into $< \kappa$ pieces} if for any $A, B \in C$, we may find associated sum embeddings $e_A^{A \oplus B}$ and $e_B^{A \oplus B}$ such that both $e_A^{A \oplus B}(A)$ and $e_B^{A \oplus B}(B)$ constitute $< \kappa$-many disjoint convex pieces in $A \oplus B$.

    If $\kappa = \aleph_0$, we say $\oplus$ \textit{splits summands finitely}.
\end{definition}

The method of shuffling ordinals we used to define the lcm sum and all its variants does not split summands finitely. That said, the Hessenberg sum always splits summands finitely. Given Carruth's characterization of natural ordinal operations, and our discussion of shuffling, one may expect it to be the unique natural sum that does so.

\begin{example}
    Given ordinals $\alpha = \sum_{i = 1} ^ n\omega^{\gamma_i} \cdot k_i$ and $\beta = \sum_{i = 1} ^ m\omega^{\delta_i} \cdot l_i$ (in CNF), define $\oplus$ inductively by

    \begin{itemize}
        \item[(i)] $\alpha \oplus 0 = 0 \oplus \alpha = \alpha$;
        \item[(ii)] If $\gamma_1 > \delta_1$, then $\alpha \oplus \beta = \beta \oplus \alpha = \alpha$;
        \item[(iii)] If $\gamma_1 = \delta_1$, then $\alpha \oplus \beta = \beta \oplus \alpha = \omega^{\gamma_1} \cdot (k_1 + l_1) + (\sum_{2 < i < n} ^{} \omega^{\gamma_i} \cdot k_i \oplus \sum_{2 < i < n} ^{}\omega^{\delta_i} \cdot l_i)$.
    \end{itemize}

    This sum is good and splits summands finitely, but is not a natural ordinal operation since, for example, $\omega\oplus1 = \omega\oplus2$.
\end{example}

\subsubsection{A good extension of the well-orders} \label{subsect: a good extension of the well-orders}

To assuage any concerns that the ordinals may be special or unique as a good class, we present a good extension of them. 

Consider the class of all orders that do not embed $\mathbb{Z}$. A concrete characterization of this class of order-types is

$$WO+WO^* = \{\alpha + \beta^*\mid\alpha, \beta \in Ord\}.$$

Let $\tilde\oplus$ be your favorite good sum on the ordinals, define its extension, $\oplus$, to all of $WO+WO^*$ as

$$(\alpha_1 + \beta_1^*) \oplus (\alpha_2 + \beta_2^*) = (\alpha_1' \tilde\oplus \alpha_2') + (n_1 + n_2) + (\beta_1'^* \tilde\oplus^{-1} \beta_2'^*),$$

where $\alpha'_1, \alpha'_2, \beta'_1, \beta'_2 \in Lim$, and for $i \in \{1, 2\}$, $\alpha_i + \beta_i^* = \alpha'_i + n_i + \beta'^*_i.$

\begin{remark} \label{rmk: reverse sum in general}
    Given any sum $\oplus$, we may define a corresponding sum $\oplus^{-1}$ by $A \oplus^{-1}B = (B^* \oplus A^*)^*$. 
    
    If $\oplus$ is an associative or commutative sum on the well-orders, $\oplus^{-1}$ will also be an associative or commutative sum on the reverse well-orders. In general, if $\oplus$ is a global sum, $\oplus^{-1}$ will preserve its properties.
\end{remark}

Since $\tilde\oplus$ is good, $\oplus$ is also good given that no instance of a sum of two limit ordinals can be a non-limit ordinal.

\subsection{Rational shuffles} \label{subsect: rational shuffles}

The example of $\mathcal{W}$ as a good class relies on the structure theory of decompositions furnished by Cantor normal form\footnote{It also depended on the freedom when defining new sums within the well-orders, since the class is closed under arbitrary sums.}. We'd like to give one more example of a good class with a concrete good sum, whose definition relies on a different piece of structure theory: Cantor's identification of the countable dense linear orders without endpoints.

\begin{theorem} \label{th: cantor's identification of Q}
    Any countable dense linear order without endpoints is isomorphic to $\mathbb{Q}$.
\end{theorem}

\begin{proof}
    By the famous back-and-forth construction.
\end{proof}

By a very similar argument we get a generalization due to Skolem.

\begin{theorem} \label{th: skolem's extension of cantor's identification of Q}
Fix some \( k \), \( 1 \leq k \leq \omega \). Let \( X, Y \) be countable dense linear orders without endpoints (two copies of $\mathbb{Q}$). Fix a partition \( X = \bigcup_{i<k} X_i \) such that each \( X_i \) is dense in \( X \), and similarly \( Y = \bigcup_{i<k} Y_i \). There is an isomorphism \( f : X \to Y \) such that \( f(X_i) = Y_i \) for every \( i < k \).
\end{theorem}

To construct our class, start with some decomposition $\mathbb{Q} = \bigcup_{n \in \omega} \mathbb{Q}$, where each $\mathbb{Q}_n$ is dense in $\mathbb{Q}$ (for example, take all the rationals which, in irreducible form, have a denominator of the form $p_{n+1}^k$, where $p_i$ is the $i^{\text{th}}$ prime, and $k \in \{1, 2, \dots\}$).

For each $n$, fix some non-empty\footnote{We may ignore all the empty orders. If you want to work with only a finite amount of orders, just repeat them.} order $I_n$, and let $\mathbb{Q}(I_n)$ be the \textit{shuffle by $I_n$}, i.e. the order that substitutes each point in $\mathbb{Q}$ that is contained in $\mathbb{Q}_n$ with $I_n$. Let $\tilde{\mathfrak{S}}$ be the class of all shuffles by $I_n$. Using Skolem's theorem, we can show that the original dense partition does not affect the order-type of these replacements.

In general, we want to let our good sum $+_{\tilde{\mathfrak{S}}}$ be such that $$\mathbb{Q}(I_n) +_{\tilde{\mathfrak{S}}} \mathbb{Q}(J_n) = \mathbb{Q}(I_n, J_n),$$ where $\mathbb{Q}(I_n, J_n)$ is the shuffle by

\begin{align*}
    K_n = 
    \begin{cases}
        I_m, \text{ if } n=2m\\
        J_m, \text{ if } n=2m+1
    \end{cases}
\end{align*}

We need some way to, given an order, recognize that it is a member of $\tilde{\mathfrak{S}}$, further we need to be able to recognize the orders $I_n$ that were used to create it. There is ambiguity as to how to do this, since there may be multiple equivalent representations of the same order, for example, $\mathbb{Q}(1) \cong \mathbb{Q}(\mathbb{Q}) \cong \mathbb{Q}(1, \mathbb{Q})$. To extract canonical representatives, we will need to avoid self-reference.

Recall that $\preceq_c$ means \textit{embeds convexly into}. Whenever we write $[\{x, y\}]$ or $(\{x, y\})$, we mean the intervals from $x$ to $y$ or $y$ to $x$, depending on whether $x \leq y$ or $y \leq x$.

\begin{definition} \label{def: self-referential condensation}
    Given a linear order $X$, define the \textit{self-referential} condensation to be the equivalence relation $\sim_X$ such that for any $x, y \in X$,
    $$x \sim_X y \text{ iff } X\not\preceq_c[\{x, y\}].$$
\end{definition}

\begin{lemma} \label{lm: self-referential condensation is a condensation}
    For any $X \in \tilde{\mathfrak{S}}$, $\sim_X$ is a condensation .
\end{lemma}

\begin{proof}
    Symmetry follows by definition, reflexivity holds since $X \not\cong 0, 1$. Convexity is also straightforward.

    For transitivity, note that if $x \sim_X y\sim_X z$, then if $X \preceq_c[\{x,z\}]$, let $X'$ be the copy of $X$ within $[\{x, z\}]$. If $y \in [\{x,z\}]$, then $y \in X'$, so if $X = \mathbb{Q}(I_n)$ and the $I_n$ segment that $y$ landed in corresponds to the rational number $q$, then the sets $(q, \infty)(I_n), (-\infty, q)(I_n) \cong \mathbb{Q}(I_n)$\footnote{Here, we replace according to the restriction of the original partition to these intervals.}, and we get a contradiction to $X \not\preceq_c[\{x,y\}]$. If not, then the setting is simpler since either one of $[\{x,y\}]$ or $[\{y,z\}]$ will contain $[\{x,z\}]$.
\end{proof}

In order to avoid problems recognizing the shuffle structure, we will say a dense rational shuffle $L$ is \textit{admissible} if $L/\sim_L$ is countable.

As expected, we're going to let the canonical shuffling orders for an admissible shuffle be the ones given by the condensation. By the homogeneity of any shuffle, it follows that these condensation classes are indeed dense in the condensation. Therefore, $L/\sim_L$ is a countable, dense linear order without endpoints (any final segment of a shuffle convexly embeds the whole order), hence isomorphic to $\mathbb{Q}$.

Using these canonical representation of the shuffles, we can formally define $+_{\tilde{\mathfrak{S}}}$.

\begin{definition} \label{def: shuffle sum}
    If $S_1, S_2$ are two admissible shuffles with two respective representations from the self-referential condensation as $\mathbb{Q}(I_n), \mathbb{Q}(J_m)$, let 
    $$S_1 +_{\tilde{\mathfrak{S}}} S_2 \cong  \mathbb{Q}(I_n, J_m).$$
\end{definition}

This sum is clearly regular and commutative, for associativity we need to show that the sum with respect to this representation yields an admissible shuffle, and that associativity indeed holds. It can be shown that this is only possible if we consider countable, not just admissible, shuffles. 

\begin{definition}
    Let $\mathfrak{S}$ be the class of all countable elements of $\tilde{\mathfrak{S}}$, and $+_\mathfrak{S}$ the restriction of $+_{\tilde{\mathfrak{S}}}$ to $\mathfrak{S}$.
\end{definition}

From now on, we will only consider shuffles in $\mathfrak{S}$.

Vacuously, every element of $\mathfrak{S}$ is admissible, and any two countable shuffles will sum to a countable shuffle, which will be admissible. To prove associativity, we'll need more information about the shuffling orders that result from the self-referential condensation.

\begin{definition} \label{def: equivalent lists of orders}
    Say two lists of countable linear orders $\{I_n\}_{n \in \omega}$ and $\{J_m\}_{m \in \omega}$ are \textit{equivalent} if for every $n \in \omega$, there is an $m \in \omega$ such that $J_m \cong I_n$, and vice-versa, if $m\in \omega$, then there is an $n \in \omega$ such that $I_n \cong J_m$.
\end{definition}

Notice that if $\{I_n\}_{n \in \omega}$ and $\{J_m\}_{m \in \omega}$ are equivalent, then $\mathbb{Q}(I_n) \cong \mathbb{Q}(J_n)$.

We abstract the key property of replacing orders emerging from the self-referential condensation.

\begin{definition} \label{def: minimal list of orders}
    A collection $\{I_n\}_{n \in \omega}$ of countable linear orders is \textit{minimal} if for all $n \in \omega$, $\mathbb{Q}_n(I_n) \not\preceq_c I_n$.
\end{definition}

Minimal orders can be characterized (up to equivalence) as those coming from the self-referential condensation.

\begin{proposition} \label{prop: uniqueness of minimal representations}
If $\{I_n\}_{n \in \omega}$ and $\{J_m\}_{m \in \omega}$ are minimal and $\mathbb{Q}(I_n) \cong \mathbb{Q}(J_m)$, then $\{I_n\}_{n \in \omega}$ and $\{J_m\}_{m \in \omega}$ are equivalent.
\end{proposition}

\begin{proof}
Fix an isomorphism $f: \mathbb{Q}(I_n) \xrightarrow{\sim} \mathbb{Q}(J_m)$. We know by density of the shuffling that any interval $I$ in $\mathbb{Q}(I_n)$ intersecting more than one of the replacing orders $I_n$ contains a convex copy of $\mathbb{Q}(I_n)$; the same is true for any interval $J$ in $\mathbb{Q}(J_m)$ intersecting more than one of the $J_m$. 

For any $I_n$, this means that $f(I_n) \subseteq J_m$ for some $m \in \omega$. We must have equality or else for some other $n' \in \omega$, $f(I_{n'}) \cap J_m \neq \emptyset$, which implies that $\mathbb{Q}(J_m) \cong \mathbb{Q}(I_n) \preceq_c J_m$. Thus, $I_n \cong J_m$. Symmetrically, for any $J_m$, there is an $I_n$ with $I_n \cong J_m$. Hence, the lists are symmetrical.
\end{proof}

It will be useful to know the relation between minimal and non-minimal representations of a shuffle.

\begin{proposition} \label{prop: relation of arbitrary list and minimal list}
Suppose that $\{I_n\}_{n \in \omega}$ is a minimal list and $\{J_m\}_{m\in \omega}$ is any other list such that $\mathbb{Q}(I_n) \cong \mathbb{Q}(J_n)$. For every $m \in \omega$, exactly one of the following holds:

\begin{itemize}
    \item[(i)] $J_m$ is isomorphic to an order in the list $\{I_n\}_{n\in \omega}$,
    \item[(ii)] $\mathbb{Q}(I_n) \preceq_c J_m$.
\end{itemize}

Moreover, in case (ii), we have $J_m \cong L + \mathbb{Q}(I_n) + R$, where $L$ is either isomorphic to an order in the list $\{I_n\}_{n \in \omega}$, or empty; and similarly $R$ is isomorphic to an order in the list $\{I_n\}_{n\in \omega}$, or empty. 
\end{proposition}

\begin{proof}
Fix $m \in \omega$ and suppose that $J_m$ is not isomorphic to any order in the list $\{I_n\}_{n \in \omega}$.

We first show $J_{m}$ convexly embeds $\mathbb{Q}(I_n)$. Suppose not. If $f: \mathbb{Q}(J_m) \xrightarrow[]{\sim} \mathbb{Q}(I_n)$, then as in the previous proposition, we must have $f(J_m) \subseteq I_n$ for some $n \in \omega$. This containment must be strict (by assumption). But then, by density, there is some other $m' \in \omega$ such that $f(J_{m'}) \cap I_n \neq \emptyset$. Let $q, r \in \mathbb{Q}$ be the rationals replaced by these orders, then 
$$\mathbb{Q}(I_n) \cong \mathbb{Q}(J_m) \cong (\{q, r\})(J_m)  \subseteq I_n,$$
in contradiction to the minimality of $I_n$.

To show that $J_m \cong L + \mathbb{Q}(I_n) + R$, there are four cases to check depending on whether or not there are minimal and maximal rationals $q$ with associated replacing order $I_{n'}$ such that $J_m \cap I_{n'} \neq \emptyset$. We carry out the case where there is such a minimal $q$, but not a maximal one.

Let $q \in \mathbb{Q}$ be the minimal such rational, and say $I_{n_q}$ is its associated replacing order. It must be the case that $I_{n_q}$ is fully contained in $J_m$, or else it intersects some other $J_{m'}$, and therefore embeds $\mathbb{Q}(I_n)$. Hence, $I_{n_q}$ is an initial segment of $J_m$.

Let $r \in \mathbb{R}$ be the supremum of all the rationals whose associated replacing orders $I_{n'}$ intersect $J_m$. This number exists, otherwise $J_m$ would be a final segment of $\mathbb{Q}(J_m)$. It follows that $J_m \cong I_{n_q} + (q, r)(I_n)$, proving the representation in this case.
\end{proof}

In case (ii.) of this proposition, we can view $J_m$ as a replacement of one of $\mathbb{Q}$, $1 + \mathbb{Q}$, $\mathbb{Q} + 1$ or $1 + \mathbb{Q} + 1$ by orders in the list $\{I_n\}$ such that each of these orders appears densely often in the replacement. The orders replacing the $1$s can be arbitrarily chosen from the list. We will call such orders \textit{extended shuffles} of the list $\{I_n\}_{n \in \omega}$. 

\begin{proposition} \label{prop: extended shuffles of minimal list lead to equivalent lists}
Suppose that $\{I_n\}_{n \in \omega}$ is a minimal list and $\mathbb{Q}(J_m)$ is a shuffle such that for every $m \in \omega$, one of the following conditions holds:
\begin{itemize}
    \item[(i)] $J_m$ is isomorphic to an order in the list $\{I_n\}_{n \in \omega}$,
    \item[(ii)] $J_m$ is an extended shuffle of the list $\{I_n\}_{n \in \omega}$.
\end{itemize}
Suppose further that for at least one $m \in \omega$, condition (ii) holds. Then $\mathbb{Q}(J_m) \cong \mathbb{Q}(I_n)$. 
\end{proposition}

Before proving the proposition, it will be helpful to briefly discuss iterated replacements.

If $Q$ is some linear order, and there is an indexed sequence of orders $\{I_q\}_{q\in Q}$, we can create a new order $Q[I_q] = \sum_{q \in \mathbb{Q}}I_q$ called the \textit{replacement of $Q$ by $\{I_q\}_{q\in Q}$}, where every point in $q \in Q$ is replaced by its corresponding point order $I_q$. Our shuffles are just a special case of this, where we replace by a sequence $\{I_q\}_{q\in \mathbb{Q}}$ such that for any $r \in \mathbb{Q}$, $\{q \in \mathbb{Q} \mid I_q \cong I_r\}$ is dense in $\mathbb{Q}$. When working with a shuffle, we may also consider its corresponding concrete representation as a replacement.

Suppose $Q$ is a linear order and $Q[I_q]$ is a replacement. If in turn each $I_q$ is a replacement $M_q[A_{(q, m)}]$, then we have $Q[I_q] = Q[M_q[A_{(q, m)}]]$. This order is naturally isomorphic to the replacement of $Q[M_q]$ by the orders $A_{(q, m)}$, i.e. we have $Q[M_q[A_{(q, m)}]] \cong Q[M_q][A_{(q, m)}]$. 

It will also be helpful to observe that if $\mathbb{Q}[M_q]$ is a replacement of $\mathbb{Q}$ such that for every $q \in \mathbb{Q}$ we have that $M_q$ is isomorphic to one of $1, \mathbb{Q}, 1 + \mathbb{Q}, \mathbb{Q} + 1$, or $1 + \mathbb{Q} + 1$, then $\mathbb{Q}[M_q] \cong \mathbb{Q}$, simply because this order is countable, dense, and without endpoints. 

\begin{proof}
Consider the sequence $\{J_q\}_{q\in \mathbb{Q}}$ such that for any $q \in \mathbb{Q}$, $J_q$ is the order $J_n$ that $q$ is replaced by in $\mathbb{Q}(J_n)$. So, $\mathbb{Q}(J_n) = \mathbb{Q}[J_q]$.

By hypothesis, we may write each $J_q$ as a replacement $M_q[I_{(q, m)}]$, where $M_q$ is one of $1, \mathbb{Q}, 1 + \mathbb{Q}, \mathbb{Q} + 1$ or $1 + \mathbb{Q} + 1$. Each $I_{(q, m)}$ is isomorphic to some $I_n$, and in the case when $M_q \neq 1$, each $I_n$ appears densely often as $I_{(q, m)}$ in the replacement $M_q[I_{(q, m)}]$. We may view $\mathbb{Q}[J_q] = \mathbb{Q}[M_q[I_{(q, m)}]]$ instead as the replacement $\mathbb{Q}[M_q][I_{(q, m)}]$ of $\mathbb{Q}[M_q] \cong \mathbb{Q}$. Thus, to show that $\mathbb{Q}(J_m) \cong \mathbb{Q}(I_n)$, it suffices to check that each $I_n$ appears densely often as $I_{(q, m)}$ in the replacement $\mathbb{Q}[M_q][I_{(q, m)}]$. 

Fix $n \in \omega$. By assumption, there is at least one $m \in \omega$ such that $J_m$ is an extended shuffle of $\{I_n\}_{n \in \omega}$, and there are densely many $q$ such that $J_q \cong J_m$. Hence, there are densely many $q$ for which $M_q \neq 1$. Fix $a < c$ in $\mathbb{Q}[M_q]$. We claim there is $b$ with $a < b < c$ such that $I_b = I_n$. If $a, c \in M_q$ for some $q$, then since $I_n$ appears densely often as $I_{(q, m)}$ in the replacement $M_q[I_{(q, m)}]$, there is such a $b$ in $M_q$. And if $a \in M_q$ and $c \in M_r$ for some $q < r$, then since there is $t$ with $q < t < r$ with $M_t \neq 1$, there is such a $b \in M_t$. We are done. 
\end{proof}

We are finally ready to prove the associativity of $+_\mathfrak{S}$.

\begin{lemma} \label{lm: normal representation of sum is a minimal representation}
Suppose that $X \cong \mathbb{Q}(I_n)$ and $Y \cong \mathbb{Q}(J_n)$ are minimal representations of $X$ and $Y$ as shuffles. Then $\mathbb{Q}(I_n, J_n)$ is a minimal representation of $X +_\mathfrak{S} Y$. 
\end{lemma}

\begin{proof}
Assume not. Then, without loss of generality, assume for some $m \in \omega$, $\mathbb{Q}(I_n, J_n) \preceq_c J_m$. If $\mathbb{Q}(K_l) \cong \mathbb{Q}(I_n, J_n)$ is a minimal representation, by Proposition \ref{prop: relation of arbitrary list and minimal list}, $J_m \cong L + \mathbb{Q}(I_n, J_n) + R$, where $L, R$ can be either one of the $K_l$, or empty. Since this will hold for any other $J_{m'}$ that convexly embeds $\mathbb{Q}(I_n, J_n)$, and if $J_{m'}$ doesn't convexly embed the sum then by Proposition \ref{prop: relation of arbitrary list and minimal list} it must be isomorphic to some $K_l$, as such we are exactly in the setting of Proposition \ref{prop: extended shuffles of minimal list lead to equivalent lists}. Therefore $\mathbb{Q}(J_n) \cong \mathbb{Q}(I_n, J_n)$, but this contradicts the minimality of $\{J_n\}_{n \in \omega}$, since then we have $\mathbb{Q}(I_n, J_n) \cong \mathbb{Q}(J_n) \preceq_c J_m$.
\end{proof}

This immediately gives us associativity, since
$$\mathbb{Q}(I_n) +_\mathfrak{S} (\mathbb{Q}(J_n) +_\mathfrak{S} \mathbb{Q}(K_n)) \cong \mathbb{Q}(I_n, J_n, K_n) \cong (\mathbb{Q}(I_n) +_\mathfrak{S} \mathbb{Q}(J_n)) +_\mathfrak{S} \mathbb{Q}(K_n).$$

Therefore, we have shown that the countable dense rational shuffles form a good class.

\subsection{Complicated classes} \label{subsect: compliacted classes}

In this section, we will present a sufficient condition for an arbitrary class of orders to be a good class. Unlike sums considered so far, the sums emerging from this result will not rely on any nice structural properties of the classes considered, but rather on their \textit{lack of structure}. This method is extremely flexible, allowing us to generate sums with practically no restrictions; the only potential drawback being a reliance on the axiom of choice.

Let $LO_\lambda = \{L \in LO \mid |L| = \lambda\}$.

\begin{definition} \label{def: complicated class}
    A class $K_\lambda \subseteq LO_\lambda$ is \textit{complicated} if given any two $A, B \in K_\lambda$, there are $2^\lambda$ non-isomorphic instances of a sum of $A$ and $B$.
\end{definition}

\begin{remark}
    Such classes could be thought of as \textit{sum-saturated}, the cardinality of possible instances of sums of two elements within the class will be the maximal one, without leaving the class.
\end{remark}

\begin{remark}
    By Theorem \ref{th: finitely many instances of sum of two ordinals}, no complicated class can contain two well-orders.
\end{remark}

Before making the sum-existence claim for complicated classes, we need to have a canonical class of representatives for the class of all order-types. Let $\mathcal{LO}_\kappa$ be a set of distinct representatives for all order-types in $\{L \in LO \mid L \text{ has universe } \kappa \in \mbox{Card}\}$

For any class of linear orders $\mathcal{C}$, let $$\overline{\mathcal{C}} = \mathcal{C} \cap \mathcal{LO}_\lambda.$$
This set has cardinality at most $2^\lambda$\footnote{Any linear order on $\lambda$ is just a subset of $\lambda \times \lambda$, therefore there are at most $|\mathcal{P}(\lambda \times \lambda)| = 2^\lambda$-many linear orders that can be imposed on $\lambda$.}.

\begin{theorem} \label{th: complicated class sum}
    Given $K_\lambda \subseteq LO_\lambda$ and any subclass $E \subseteq K_\lambda$, such that $K_\lambda - E$ is complicated, and there is a good sum $\oplus_E$ on $E$. If $<_\lambda$ is a well-order on $\overbar{K_\lambda}$, then exists is a good sum $\oplus$ on $K_\lambda$ extending $\oplus_E$.
\end{theorem}

\begin{proof}
    Let $\langle L_\alpha \mid \alpha < 2^\lambda \rangle$ be the enumeration of $\overbar{K_\lambda}$ furnished by $<_\lambda$, and take $\mathcal{B}_E$ to consist of the ordinals assigned to representatives of types in $E$ by $<_\lambda$.

    By induction on $\alpha$, define $(\mathcal{U}_\alpha, \oplus_\alpha)$ such that:

    \begin{enumerate}
        \item[(a)] $\mathcal{U}_\alpha \subseteq \alpha\cup\mathcal{B}_E$;
        \item[(b)] If $\beta < \alpha$, then $\mathcal{U}_\beta = \mathcal{U}_\alpha \cap (\beta \cup \mathcal{B}_E)$;
        \item[(c)] For all $n < \omega$ and $\langle \beta_0, \dots, \beta_n \rangle \in \prescript{n + 1}{}{(\alpha\cup\mathcal{B}_E)}$, we define $\oplus_\alpha \langle L_{\beta_l} \mid l \leq n \rangle$ (a sum defined on sequences, with outputs in $\overbar{K_\lambda}$) with the following properties:

        \begin{enumerate}
            \item[\textbullet\textsubscript{1}] $\oplus_\alpha (\langle L_\beta \rangle) = L_\beta$, if $\beta \in \alpha\cup\mathcal{B}_E$;

            \item[\textbullet\textsubscript{2}] If $\beta < \alpha$, then $\oplus_\beta = \oplus_\alpha\restriction\{\langle L_{\beta_l} \mid l \leq n \rangle \mid \langle \beta_l \mid l \leq n \rangle \in \prescript{n+1}{}{(\beta \cup \mathcal{B}_E)}\}$;

            \item[\textbullet\textsubscript{3}] For all $l \leq m, n$, if $\beta_l, \gamma_l \in \mathcal{U}_\alpha\cap\alpha$ and $\oplus_\alpha \langle L_{\beta_0}, \dots, L_{\beta_n} \rangle \cong \oplus_\alpha \langle L_{\gamma_0}, \dots, L_{\gamma_m} \rangle$, then $n = m$ and $\langle L_{\beta_0}, \dots, L_{\beta_n} \rangle$ is a permutation of $\langle L_{\gamma_0}, \dots, L_{\gamma_m} \rangle$. Likewise, if $\delta, \varepsilon \in \mathcal{B}_E$ and $\oplus_\alpha \langle L_{\beta_0}, \dots, L_\delta \dots, L_{\beta_n} \rangle \cong \oplus_\alpha \langle L_{\gamma_0}, \dots, L_\varepsilon, \dots, L_{\gamma_m} \rangle$, then $n = m$, $\delta = \varepsilon$, and $\langle L_{\beta_0}, \dots, L_\delta \dots, L_{\beta_n} \rangle$ is a permutation of $\langle L_{\gamma_0}, \dots, L_\varepsilon, \dots, L_{\gamma_m} \rangle$. The sum of a sequence containing an element of $E$ is never isomorphic to one of elements of $K_\lambda - E$;

            \item[\textbullet\textsubscript{4}] If $\gamma \in (\alpha \cup \mathcal{B}_E)$, then for some $\beta_0, \dots, \beta_n \in \mathcal{U}_\alpha$, $L_\gamma \cong \oplus_\alpha \langle L_{\beta_0}, \dots, L_{\beta_n}\rangle$, where at most one of the indices is in $\mathcal{B}_E$;

            \item[\textbullet\textsubscript{5}] For any $\beta_0, \dots, \beta_n \in \alpha$, if $\sigma \in S_{n+1}$, $$\oplus_\alpha \langle L_{\beta_0}, \dots, L_{\beta_n} \rangle \cong \oplus_\alpha \langle L_{\beta_{\sigma(0)}}, \dots, L_{\beta_{\sigma(n)}} \rangle.$$
        \end{enumerate}
    \end{enumerate}
    
    Let $\mathcal{U}_0 = \mathcal{B}_E$, $\oplus_0 = \oplus_E$.
    
    For $\alpha+1$, let $\mathcal{U}_{\alpha+1} = \mathcal{U}_{\alpha}$ if $L_\alpha$ is in the range of $\oplus_\alpha$, and $\mathcal{U}_{\alpha} \cup \{\alpha\}$ otherwise. To satisfy condition (c).2, let $\oplus_{\alpha+1}\restriction\{\langle L_{\beta_l} \mid l \leq n \rangle \mid \langle \beta_l \mid l \leq n \rangle \in \prescript{n+1}{}{(\alpha \cup \mathcal{B}_E)}\} =\oplus_\alpha$. If $L_\alpha$ is in the range of $\oplus_\alpha$, let $\oplus_{\alpha+1}$ remain consistent with $\oplus_{\alpha}$ by adding any tuple containing $L_\alpha$ as $\oplus_\alpha$ would, by expanding any instances of $L_\alpha$ as tuple of orders in $\{L_\beta \mid \beta\in \mathcal{U}_\alpha\}$ which sum to it. 
    
    If the new order is not in the range, enumerate all unordered tuples (with repetition) of $\mathcal{U}_{\alpha+1}$ with at most one element in $\mathcal{B}_E$\footnote{We may consider unordered tuples by (c).3, (c).5. If one of the indices is not in $\mathcal{U}_\alpha$, expand it to a tuple in $\mathcal{U}_\alpha$ using (c).4. Collapse elements of $\mathcal{B}_E$ using $\oplus_0 = \oplus_E$.} lexicographically, prioritizing shorter tuples. Inductively, define our sum on these tuples. For tuples of length $1$, follow (c).1. Given a tuple of length $>1$, if $\alpha$ is in it, using the inductive hypothesis, collapse the tuple until it is length $2$, and let the sum of the tuple be any $L_\gamma$ for $\gamma > \alpha$ such that $L_\gamma$ is an instance of a sum of the two orders considered, which is not in the range of $\oplus_\alpha$ nor $\oplus_{\alpha+1}$ so far (we will never run out of these instances by the assumption that $K_\lambda - E$ is complicated). Note that the inclusion of $\alpha$ into $\mathcal{U}_{\alpha+1}$ guarantees that (c).4 is satisfied. Since we mandated that the new orders we map to are unique to each unordered tuple, (c).3 is satisfied.

    For limit stages, let $\mathcal{U}_\lambda =\bigcup_{\alpha < \lambda} \mathcal{U}_\alpha$, and take $\oplus_\lambda = \bigcup_{\alpha<\lambda}\oplus_\alpha$.

    The operation $\oplus = \bigcup_{\alpha < 2^\lambda}\oplus_\alpha$ restricted to pairs and extended to all of $K_\lambda$ by respecting regularity, is a good sum.
\end{proof}

\begin{remark}
    The condition that we have a good subclass $E$ is not restrictive, if we let $E = \emptyset$, we get the a sufficient condition for $K_\lambda$ to be a good class. As we will see though, the allowance for extension gives us a vast amount of flexibility.
\end{remark}

But do complicated classes even exist in general? 

\begin{example} \label{ex: binary shuffle complicated class}
    Let $B_{\aleph_0}$ be the class of all countable orders with $\mathbb{Q}$ as a convex suborder.

    All we need to show here is that there are $2^{\aleph_0}$ ways to combine $\mathbb{Q}$ with itself while remaining within this class. The trick to do this will be by showing that we can represent any infinite binary sequence as an instance of a sum. First, we show that we can combine $\mathbb{Q}$ with itself to make any of

    $$\mathbb{Q}, \mathbb{Q}(2), \mathbb{Q}(1,2)$$

    Effectively we will use these to encode our binary sequences, with $\mathbb{Q}$ acting as a spacer, $\mathbb{Q}(2)$ as $0$, and $\mathbb{Q}(1, 2)$ as $1$.

    Note that any $\omega$-sum of these orders will itself be an instance of a sum of two copies of the rationals. First, we show how to decompose all of these orders into two copies of the rationals, a ``red'' copy and a ``blue'' copy of which they are an instance of a sum. For $\mathbb{Q}$, let the red copy be $(-\infty, \pi) \cap \mathbb{Q}$, and the blue copy $(\pi, \infty) \cap \mathbb{Q}$. From $\mathbb{Q}(2)$, let the red copy be all successors, and the blue one to be all predecessors. Finally, for $\mathbb{Q}(1, 2)$, consider all predecessors and the points without predecessors or successors to be the red points, and the set of all successors to be the blue ones. Given an $\omega$-sum of these orders, note that the sum of all red orders, and the sum of all blue orders, are still countable DLOs without endpoints, therefore isomorphic to the rationals, hence the sum is an instance of a sum of two copies of the rationals.

    Since we can encode any countable binary sequence with two elements in $B_{\aleph_0}$, this class is complicated.
\end{example}

\begin{example} \label{ex: usable compliacted class}
    Consider $\mathcal{Q} = \langle \mathbb{Q}, \mathbb{Q}(1, 2) + \mathbb{Q}, \mathbb{Q}(2) + \mathbb{Q} \rangle$. 

This class is the closure of $$\{\mathbb{Q}, 1 + \mathbb{Q}, \mathbb{Q}(1, 2) + \mathbb{Q}, 1 + \mathbb{Q}(1, 2) + \mathbb{Q}, 2 + \mathbb{Q}(1, 2) + \mathbb{Q}, \mathbb{Q}(2) + \mathbb{Q}, 1 + \mathbb{Q}(2) + \mathbb{Q}, 2 + \mathbb{Q}(2) + \mathbb{Q}\}$$ under sums indexed by ordinals. So it is not hard to see that, by a similar approach to that of Example \ref{ex: binary shuffle complicated class}, we may encode binary sequences of length $\kappa$ as instances of sums of any two elements of $\mathcal{Q}_\kappa$. Therefore, $\mathcal{Q}_\kappa$ is $\kappa$-complicated.
\end{example}

\begin{definition} \label{def: global complicated class}
    A class $K$ of linear orders is \textit{complicated} if $K_\lambda$ is complicated for every $\lambda \in Card$.
\end{definition}

$\mathcal{Q}$ is our first example of a complicated with orders of arbitrary (infinite) cardinality.

\section{Revisiting the properties of sums} \label{sect: revisiting the properties of sums}

Having developed a robust theory for constructing new sums, we are now in a position to revisit our two original properties of sums, point-tracking and canonical regularity. We will show that neither of them follow from regularity paired with associativity. To illustrate the flexibility of our methods of constructing sums, we will also consider whether any global sum can represent a non-trivial group, showing that none of the structured sums we have considered so far can, while constructing a sum that represent $\mathbb{Z}/2\mathbb{Z}$.

The guiding heuristic for this section will be that whenever there is a reasonable property of sums that does not dictate rigid structure (compare properties like point-tracking and commutativity with being semi-standard), there will be a dichotomy between the structured sums, like those built from ordinal good sums that split summands finitely and simple sums, and the non-structured sums, such as those constructed from complicated classes.

\subsection{Point-tracking} \label{subsect: point-tracking}

Recall that the original point-tracking condition from Definition \ref{def: point-tracking} ensures that associativity is witnessed by an isomorphism that makes it such that the joint order between summands is the same in both instances of associativity is preserved. This is an inherently local property, and in cases where the global structure of our sums is nice, it is straightforward to show.

\begin{proposition} \label{prop: simple sums track points}
    Simple sums track points.
\end{proposition}

\begin{proof}
    The induced $f: (A \oplus B) \oplus C \xrightarrow[]{\sim} A \oplus (B \oplus C)$ isomorphism will just be the composition of the usual isomorphisms witnessing
    $$(A \oplus B) \oplus C \cong C_L + B_L + A + B_R + C_R \cong A \oplus (B \oplus C).$$

    Verify that this works by running through commutative diagram using that $$A \oplus B \cong B_L + A + B_R, \;B \oplus C = C_L + B + C_R.$$
\end{proof}

That said, for sums which shuffle points more intrusively, it can fail.

Let $+_{lcm}$ denote the global extension of the lcm sum (the min or dynamic sum would do just as well) note that 
    $$\omega +_{lcm} (\omega +_{lcm}\omega) \cong (\omega +_{lcm} \omega) +_{lcm} \omega$$
    is a counterexample to point-tracking. By using the standard embeddings associated to this sum, we can prove this visually:

    \begin{tikzpicture}[scale=1.2, thick]
\node at (-0.5, 0) {$($};

\foreach \x in {0, 1,2,3} {
  \draw[fill={rgb,255:red,194;green,106;blue,119}, draw=none] (\x/2,0) circle (3pt);
}

\node at (2,0) {\dots};

\node at (2.5,0) {$+_{lcm}$};

\foreach \x in {6, 7, 8, 9} {
  \draw[fill={rgb,255:red,93;green,168;blue,153},draw=none] (\x/2,0) circle (3pt);
}

\node at (5,0) {\dots};

\node at (5.5,0) {$)$};

\node at (6,0) {$+_{lcm}$};

\foreach \x in {13, 14, 15, 16} {
  \draw[fill={rgb,255:red,204;green,187;blue,68},draw=none] (\x/2,0) circle (3pt);
}
\node at (8.5,0) {\dots};

\node at (-1, -0.5) {$=$};
\node at (-0.5, -0.5) {$($};

\foreach \x in {0, 1,2,3} {
  \draw[fill={rgb,255:red,194;green,106;blue,119}, draw=none]  (\x,-0.5) circle (3pt);
  \draw[fill={rgb,255:red,93;green,168;blue,153},draw=none] (\x + 0.5,-0.5) circle (3pt);
}

\node at (4,-0.5) {\dots};

\node at (4.5,-0.5) {$)$};

\node at (5, -0.5) {$+_{lcm}$};

\foreach \x in {11, 12, 13, 14} {
  \draw[fill={rgb,255:red,204;green,187;blue,68},draw=none] (\x/2,-0.5) circle (3pt);
}
\node at (7.5,-0.5) {\dots};

\node at (-1, -1) {$=$};

  \draw[fill={rgb,255:red,194;green,106;blue,119}, draw=none]  (0,-1) circle (3pt);
  \draw[fill={rgb,255:red,204;green,187;blue,68},draw=none] (0.5,-1) circle (3pt);
  \draw[fill={rgb,255:red,93;green,168;blue,153},draw=none] (1,-1) circle (3pt);
    \draw[fill={rgb,255:red,204;green,187;blue,68},draw=none] (1.5,-1) circle (3pt);
  \draw[fill={rgb,255:red,194;green,106;blue,119}, draw=none]  (2,-1) circle (3pt);
  \draw[fill={rgb,255:red,204;green,187;blue,68},draw=none] (2.5,-1) circle (3pt);
    \draw[fill={rgb,255:red,93;green,168;blue,153},draw=none] (3, -1) circle (3pt);
  \draw[fill={rgb,255:red,204;green,187;blue,68},draw=none] (3.5,-1) circle (3pt);
  \draw[fill={rgb,255:red,194;green,106;blue,119}, draw=none]  (4,-1) circle (3pt);
    \draw[fill={rgb,255:red,204;green,187;blue,68},draw=none] (4.5,-1) circle (3pt);
  \draw[fill={rgb,255:red,93;green,168;blue,153},draw=none] (5,-1) circle (3pt);
  \draw[fill={rgb,255:red,204;green,187;blue,68},draw=none] (5.5,-1) circle (3pt);

\node at (6,-1) {\dots};
\end{tikzpicture}

But, the other side of associativity yields:

\begin{tikzpicture}[scale=1.2, thick]

\foreach \x in {-1, 0,1,2} {
  \draw[fill={rgb,255:red,194;green,106;blue,119}, draw=none]  (\x/2,0) circle (3pt);
}

\node at (1.5,0) {\dots};

\node at (2,0) {$+_{lcm}$};

\node at (2.5,0) {$($};

\foreach \x in {6, 7, 8, 9} {
  \draw[fill={rgb,255:red,93;green,168;blue,153},draw=none] (\x/2,0) circle (3pt);
}

\node at (5,0) {\dots};

\node at (5.5,0) {$+_{lcm}$};

\foreach \x in {12, 13, 14, 15} {
  \draw[fill={rgb,255:red,204;green,187;blue,68},draw=none] (\x/2,0) circle (3pt);
}
\node at (8,0) {\dots};

\node at (8.5, 0) {$)$};

\node at (-1, -0.5) {$=$};

\foreach \x in {-1, 0,1,2} {
  \draw[fill={rgb,255:red,194;green,106;blue,119}, draw=none]  (\x/2,-0.5) circle (3pt);
}

\node at (1.5,-0.5) {\dots};

\node at (2,-0.5) {$+_{lcm}$};

\foreach \x in {2.5, 3.5,4.5,5.5} {
  \draw[fill={rgb,255:red,93;green,168;blue,153},draw=none] (\x,-0.5) circle (3pt);
  \draw[fill={rgb,255:red,204;green,187;blue,68},draw=none] (\x + 0.5,-0.5) circle (3pt);
}

\node at (6.5,-0.5) {\dots};

\node at (-1, -1) {$=$};

  \draw[fill={rgb,255:red,194;green,106;blue,119}, draw=none]  (-0.5,-1) circle (3pt);
  \draw[fill={rgb,255:red,93;green,168;blue,153},draw=none] (0,-1) circle (3pt);
  \draw[fill={rgb,255:red,194;green,106;blue,119}, draw=none]  (0.5,-1) circle (3pt);
    \draw[fill={rgb,255:red,204;green,187;blue,68},draw=none] (1,-1) circle (3pt);
  \draw[fill={rgb,255:red,194;green,106;blue,119}, draw=none]  (1.5,-1) circle (3pt);
  \draw[fill={rgb,255:red,93;green,168;blue,153},draw=none] (2,-1) circle (3pt);
    \draw[fill={rgb,255:red,194;green,106;blue,119}, draw=none]  (2.5, -1) circle (3pt);
  \draw[fill={rgb,255:red,204;green,187;blue,68},draw=none] (3,-1) circle (3pt);
  \draw[fill={rgb,255:red,194;green,106;blue,119}, draw=none]  (3.5,-1) circle (3pt);
    \draw[fill={rgb,255:red,93;green,168;blue,153},draw=none] (4,-1) circle (3pt);
  \draw[fill={rgb,255:red,194;green,106;blue,119}, draw=none]  (4.5,-1) circle (3pt);
  \draw[fill={rgb,255:red,204;green,187;blue,68},draw=none] (5,-1) circle (3pt);

\node at (5.5,-1) {\dots};
\end{tikzpicture}

As we can see, the maximal shuffling of this sums makes it so that any two points in the right two orders that have the same positions in their respective copies of $\omega$ get their orders flipped with respect to each other among instances of associativity. This contradicts the existence of a point-tracking isomorphism.

\subsection{Canonical regularity} \label{subsect: canonical regularity}

As a direct consequence of the rigidity theorem, every simple sum is canonically regular (sum-generated components are rigid under isomorphism).

Based on the previous example we had of a regular sum that is not canonically regular (Example \ref{ex: regular sum which is not canonically regular}), we may expect every regular associative sum to be canonically regular given that regular sums which aren't canonically regular have to shuffle their summands with little robustness.

Using the method of complicated classes, we show that this intuition is false.

Start by defining $\oplus_E$ on $E = \{2 + \mathbb{Q}(2) + \mathbb{Q}, \mathbb{Q}, 1 + \mathbb{Q} + \mathbb{Q}(2) + \mathbb{Q}\}$ by letting $\oplus_E$ be commutative, regular and satisfy
$$(2 + \mathbb{Q}(2) + \mathbb{Q}) \oplus_E \mathbb{Q} = 1 + \mathbb{Q} + \mathbb{Q}(2) + \mathbb{Q}$$ 
$$(2 + \mathbb{Q}(2) + \mathbb{Q}) \oplus_E (2 + \mathbb{Q}(2) + \mathbb{Q}) = 2 + \mathbb{Q}(2) + \mathbb{Q}$$ 
$$(2 + \mathbb{Q}(2) + \mathbb{Q}) \oplus_E (1 + \mathbb{Q} + \mathbb{Q}(2) + \mathbb{Q}) = 1 + \mathbb{Q} + \mathbb{Q}(2) + \mathbb{Q}$$ 
$$(1 + \mathbb{Q} + \mathbb{Q}(2) + \mathbb{Q}) \oplus_E \mathbb{Q} = 1 + \mathbb{Q} + \mathbb{Q}(2) + \mathbb{Q}$$
$$(1 + \mathbb{Q} + \mathbb{Q}(2) + \mathbb{Q}) \oplus_E (1 + \mathbb{Q} + \mathbb{Q}(2) + \mathbb{Q}) = 1 + \mathbb{Q} + \mathbb{Q}(2) + \mathbb{Q}$$
$$\mathbb{Q} \oplus_E \mathbb{Q} = \mathbb{Q}.$$

This is a possible set of outputs of a sum which make $\oplus_E$ associative on $E$.

Note that the only way to arrange

$$(2 + \mathbb{Q}(2) + \mathbb{Q}) \oplus_E \mathbb{Q} = 1 + \mathbb{Q} + \mathbb{Q}(2) + \mathbb{Q}$$

is by splitting the right summand, $\mathbb{Q}$, into at least two pieces, placing a piece of the form $(-\infty, x)$ between $0$ and $1$ in the $2$, while some segment $(x, y)$ goes immediately to the right of $1 \in 2$. Here $x$ must be irrational, and $y$ has to be either irrational or $\infty$.

With the notation from definition 5, let $\psi: \mathbb{Q} \xrightarrow[]{\sim} \mathbb{Q}$ be such that $q \mapsto q+1$, and $\phi: 2 + \mathbb{Q}(2) + \mathbb{Q} \xrightarrow[]{\sim} 2 + \mathbb{Q}(2) + \mathbb{Q}$ be the identity. Assume this sum is canonically regular, if $f: 1 + \mathbb{Q} + \mathbb{Q}(2) + \mathbb{Q} \xrightarrow[]{\sim} 1 + \mathbb{Q} + \mathbb{Q}(2) + \mathbb{Q}$ is the induced automorphism and $q_x \in (x-1, x)$, while $$e_{\mathbb{Q}}^{(2 + \mathbb{Q}(2) + \mathbb{Q}) \oplus_E \mathbb{Q}}(q_x) < e_{1+\mathbb{Q}(2)+\mathbb{Q}}^{(2 + \mathbb{Q}(2) + \mathbb{Q}) \oplus_E \mathbb{Q}}(1)$$
where $1$ here represents the $1 \in 2$ to the left of $2+\mathbb{Q}(2)+\mathbb{Q}$. As shown above,

$$f(e_{\mathbb{Q}}^{(2 + \mathbb{Q}(2) + \mathbb{Q}) \oplus_E \mathbb{Q}}(q_x)) = e_{\mathbb{Q}}^{(2 + \mathbb{Q}(2) + \mathbb{Q}) \oplus_E \mathbb{Q}}(q_x + 1) > e_{1+\mathbb{Q}(2)+\mathbb{Q}}^{(2 + \mathbb{Q}(2) + \mathbb{Q}) \oplus_E \mathbb{Q}}(1) = f(e_{1+\mathbb{Q}(2)+\mathbb{Q}}^{(2 + \mathbb{Q}(2) + \mathbb{Q}) \oplus_E \mathbb{Q}}(1)).$$

This is a contradiction to $f$ being order-preserving.

Extend this sum to a global sum by first using the complicated class construction to extend $\oplus_E$ to a sum $\oplus_0$ on all of $\mathcal{Q}_{\aleph_0}$. Noting that $\mathcal{Q}_{\aleph_0} = \hat{\mathcal{Q}}^{card}_0$, use the cardinality sifting scheme $\{(\mathcal{Q}^{card}_\alpha, \oplus_\alpha)\}_{\alpha \in Ord}$, where $\oplus_\alpha = +$ for $\alpha > 0$. $\mathcal{Q}$ is generated by a finite set of countable orders closed under non-empty final segments, therefore this filtration will be effective, and the resulting sifted sum $\oplus$ will be regular, associative and will extend $\oplus_E$, therefore it won't be canonically regular.

To recapitulate: we had to break canonical regularity, a local property. To do so, we started out by explicitly writing down a specific failure of canonical regularity, and subsequently built a good sum on a small set around it. This gave us enough consistency to extend the sum (using the complicated class construction) to the subclass of countable orders of a large complicated class. Leveraging the fact that this complicated class is nice since it is also an SGC and behaves well with respect to its associated cardinality filtration, we used sifting to extend this sum on the subclass of countable orders, to a global sum.

\subsection{Representing groups} \label{subsect: repressenting groups}

Let $\mathcal{LO}$ be the class of order-types.

\begin{definition} \label{def: representing a group}
    Given a group $(G, \cdot)$, we say an isomorphism invariant binary operation $\star: \mathcal{LO} \times \mathcal{LO} \rightarrow \mathcal{LO}$ \textit{represents} $G$ if there exists a one-to-one mapping $\phi: G \rightarrow\mathcal{LO}$ such that for any $g, h \in G$,

    $$\phi(g\cdot h) = \phi(g) \star \phi(h).$$

    If there exists some group $(G, \cdot)$ such that $\star$ represents $G$, we say $(\mathcal{LO}, \star)$ \textit{represents a group}. 
\end{definition}

All structured sums we have considered so far do not represent any non-trivial groups.

\begin{proposition} \label{prop: simple sums don't represent groups}
    If $\oplus$ is a simple sum, the only group that can be represented in $(\mathcal{LO}, \oplus)$ is the trivial group.
\end{proposition}

\begin{proof}
    Say $\phi: G \rightarrow \mathcal{LO}$ is a representation of a group $(G, \cdot)$, and let $e \in G$ be the identity, and $a \in G$ be any element. Since $$\phi(e)_L + \phi(a) + \phi(e)_R = \phi(a) \oplus \phi(e) = \phi(a \cdot e) = \phi (a),$$  we must have $\phi(e)_L + \phi(a)_L = \phi(a)_L$ by the rigidity theorem. Similarly, $$\phi(e) = \phi(a\cdot a^{-1}) = \phi(a) \oplus \phi(a^{-1}) = \phi(a^{-1})_L + \phi(a) + \phi(a^{-1})_R,$$ so $\phi(e)_L = \phi(a^{-1})_L + \phi(a)_L$. By the Fundamental Theorem, $\phi(e)_L = \phi(a)_L$. 
    
    The exact same procedure may be carried out to show that $\phi(a)_R = \phi(a)_R$, thus $\phi(e) = \phi(a)$, so $e = a$. Since $a$ was arbitrary, this shows that $G$ is trivial.
\end{proof}

With a simple application of the bounds for sums of ordinals we found in \ref{subsect: good sums of well-orders}, we can also discard the case of sums on the ordinals given the analogous definition for representing sums in $(Ord, \oplus)$.

\begin{proposition} \label{prop: ordinal sums don't represent groups}
  Given a sum $\oplus$, the only group that can be represented in $(Ord, \oplus)$ is the trivial group.
\end{proposition}

\begin{proof}
    For any two ordinals $\alpha, \beta$, $\alpha \oplus \beta \geq \alpha +_{min} \beta \geq \max\{\alpha, \beta\}$. 
    
    If $\varepsilon$ represents the identity element, $\alpha$ represents some non-trivial element and $\alpha^{-1}$ its inverse, then
    $$\alpha = \varepsilon \oplus \alpha \geq \max\{\varepsilon, \alpha\}$$
    $$\alpha^{-1} = \varepsilon \oplus \alpha^{-1} \geq \max\{\varepsilon, \alpha^{-1}\}.$$
    So, $\alpha, \alpha^{-1} > \varepsilon$, but
    $$\varepsilon = \alpha \oplus \alpha^{-1} \geq \max\{\alpha, \alpha^{-1}\}$$
    a contradiction.
\end{proof}

At this point, it seems like we've exhausted most reasonable sums we usually work with, they're just too rigid! 

That said, noticing that representing a non-trivial group is a local property, we can use the more fluid complicated classes to ``force'' a representation of a non-trivial group.

Working in $\mathcal{Q}$, let our base set be $E = \{\mathbb{Q}(2) + \mathbb{Q}, \mathbb{Q}(1,2) + \mathbb{Q}\}$. Define the sum $\oplus_E$ on $E$ by

$$(\mathbb{Q}(2) + \mathbb{Q})\oplus_E (\mathbb{Q}(2) + \mathbb{Q}) = \mathbb{Q}(2) + \mathbb{Q}$$ 
$$(\mathbb{Q}(2) + \mathbb{Q}) \oplus_E (\mathbb{Q}(1,2) + \mathbb{Q}) = (\mathbb{Q}(1,2) + \mathbb{Q}) \oplus_E (\mathbb{Q}(2) + \mathbb{Q}) = \mathbb{Q}(1,2) + \mathbb{Q}$$ 
$$(\mathbb{Q}(1,2) + \mathbb{Q}) \oplus_E (\mathbb{Q}(1,2) + \mathbb{Q}) = \mathbb{Q}(2) + \mathbb{Q}.$$

It is not hard to check that these solutions are instances of sums of the pairs given and that $\oplus_E$ is a good sum on $E$ (this is just the multiplication table for $\mathbb{Z}/2\mathbb{Z}$). Extend $\oplus_E$ to a global sum $\oplus$ by following the same steps we did for our failure of canonical regularity.

The mapping $\phi: \mathbb{Z}/2\mathbb{Z} \rightarrow \mathcal{LO}$ such that $$\phi([0]) = \mathbb{Q}(2) + \mathbb{Q}, \; \phi([1]) = \mathbb{Q}(1,2) + \mathbb{Q}$$
is a representation of $\mathbb{Z}/2\mathbb{Z}$ in $(\mathcal{LO}, \oplus)$.

\printbibliography

\end{document}